\documentclass{article}
\usepackage[top=25mm,bottom=35mm,left=30mm,right=30mm]{geometry}

\usepackage[dvipdfmx]{graphicx}
\usepackage{amsmath,amssymb,amscd,amsthm,bm}
\usepackage{color}
\usepackage[all]{xy}

\newtheorem{thm}{Theorem}[section]
\newtheorem{lem}[thm]{Lemma}
\newtheorem{prop}[thm]{Proposition}

\newtheorem{exam}[thm]{Example}

\makeatletter 
\@tempcnta\z@
\loop\ifnum\@tempcnta<26
\advance\@tempcnta\@ne
\expandafter\edef\csname f\@Alph\@tempcnta\endcsname{\noexpand\mathfrak{\@Alph\@tempcnta}}
\expandafter\edef\csname l\@Alph\@tempcnta\endcsname{\noexpand\mathbb{\@Alph\@tempcnta}}
\expandafter\edef\csname c\@Alph\@tempcnta\endcsname{\noexpand\mathcal{\@Alph\@tempcnta}}
\expandafter\edef\csname b\@Alph\@tempcnta\endcsname{\noexpand\mathbf{\@Alph\@tempcnta}}
\repeat

\newcommand{\dd}{{\delta}}
\newcommand{\DD}{{\Delta}}
\newcommand{\ee}{{\varepsilon}}

\newcommand{\ra}{{\rightarrow}}

\newcommand{\Pers}{{\rm Pers}}
\newcommand{\pers}{{\rm pers}}

\newcommand{\Lip}{{\rm Lip}}
\newcommand{\Amp}{{\rm Amp}}
\newcommand{\diam}{{\rm diam}}
\newcommand{\mesh}{{\rm mesh}}

\newcommand{\median}{{\rm median}}
\newcommand{\KFDR}{{\rm KFDR}}

\newcommand{\lmid}{ \ \middle| \ }

\providecommand{\abs}[1]{\left\lvert#1\right\rvert}
\providecommand{\norm}[1]{\left\lVert#1\right\rVert}
\providecommand{\pare}[1]{\left( #1 \right)}
\providecommand{\rl}[1]{\left\{ #1 \right\}}
\providecommand{\card}[1]{{\rm card}\pare{#1}}

\providecommand{\inn}[2]{\langle #1, #2 \rangle}

\title{Persistence weighted Gaussian kernel for topological data analysis}
\author{Genki Kusano \thanks{Tohoku University, genksn@gmail.com}
\and Kenji Fukumizu \thanks{The Institute of Statistical Mathematics, fukumizu@ism.ac.jp}
\and Yasuaki Hiraoka \thanks{Tohoku University, hiraoka@wpi-aimr.tohoku.ac.jp}}
\date{}
\begin{document}
\maketitle

\begin{abstract} 
Topological data analysis (TDA) is an emerging mathematical concept for characterizing shapes in complex data. In TDA, persistence diagrams are widely recognized as a useful descriptor of data, and can distinguish robust and noisy topological properties.  This paper proposes a kernel method on persistence diagrams to develop a statistical framework in TDA. The proposed kernel satisfies the stability property and provides explicit control on the effect of persistence. Furthermore, the method allows a fast approximation technique. The method is applied into practical data on proteins and oxide glasses, and the results show the advantage of our method compared to other relevant methods on persistence diagrams.
\end{abstract}

\section{Introduction}
Recent years have witnessed an increasing interest in utilizing methods of algebraic topology for statistical data analysis. This line of research is called {\em topological data analysis} (TDA) \cite{Ca09}, which has been successfully applied to various areas including information science \cite{CIdSZ08,dSG07}, biology \cite{KZPSGP07,XW14}, brain science \cite{LCKKL11,PETCNHV14,SMISCR08}, biochemistry \cite{GHIKMN13}, and material science \cite{NHHEMN15, NHHEN15}.  In many of these applications, it is not straightforward to provide feature vectors or descriptors of data from their complicated geometric configurations.  The aim of TDA is to detect informative topological properties (e.g., connected components, rings, and cavities) from such data, and use them as descriptors.

A key mathematical apparatus in TDA is {\em persistent homology}, which is an algebraic method for extracting robust topological information from data.  To provide some intuition for the persistent homology, let us consider a typical way of constructing persistent homology from data points in a Euclidean space, assuming that the data lie on a submanifold.  The aim is to make inference on the topology of the underlying manifold from finite data.  We consider the $r$-balls (balls with radius $r$) to recover the topology of the manifold, as popularly employed in constructing an $r$-neighbor graph in many manifold learning algorithms. While it is expected that, with an appropriate choice of $r$, the $r$-ball model can represent the underlying topological structures of the manifold, it is also known that the result is sensitive to the choice of $r$.  If $r$ is too small, the union of $r$-balls consists simply of the disjoint $r$-balls. On the other hand, if $r$ is too large, the union becomes a contractible space. {\em Persistent homology} \cite{ELZ02} can consider {\em all} $r$ simultaneously, and provides an algebraic expression of topological properties together with their persistence over $r$.  We give a brief explanation of persistent homology in Supplementary material A.3.
\begin{figure}[htbp]
 \begin{center}
 \vspace{-3mm}
  \includegraphics[width=70mm]{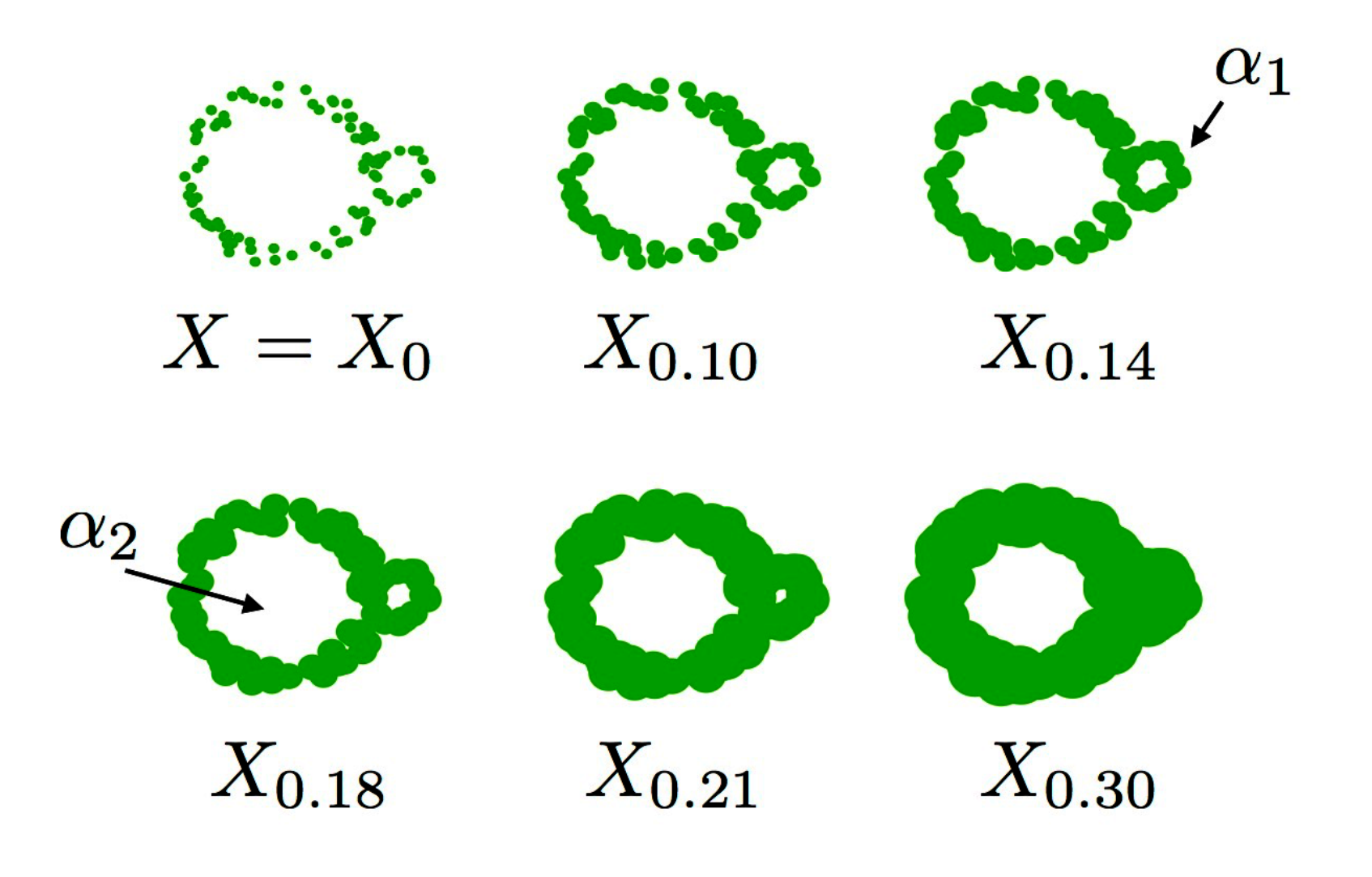}
  \vspace{-5mm}
  \caption{The union $ X_{r}$ of $r$-balls at points sampled from annuli with noise. }
 \end{center}
  \label{fig:filtration}
\end{figure}

The persistent homology can be visualized in a compact form called a {\em persistence diagram} $D=\{(b_i,d_i) \in \lR^{2} \mid i\in I, \ b_i \leq d_i\}$, and this paper focuses on persistence diagrams, since the contributions of this paper can be fully explained in terms of persistence diagrams. 
Every point $(b_i,d_i)\in D$, called a {\em generator} of the persistent homology, represents a topological property (e.g., connected components, rings, and cavities) which appears at $X_{b_i}$ and disappears at $X_{d_i}$ in the $r$-ball model. Then, the {\em persistence} $d_i-b_i$ of the generator shows the robustness of the topological property under the radius parameter. As an example shown in Figure \ref{fig:filtration}, the rings $\alpha_{1},\alpha_{2}$ and other tiny ones are expressed as $x_{1},x_{2}$, and the other points in the persistence diagram shown in Figure \ref{fig:pd}. A topological property with large persistence can be regarded as a reliable structure, while that with small persistence (points close to the diagonal) is likely to be noise. In this way, persistence diagrams encode topological and geometric information of data points.

\begin{figure}[tbp]
 \begin{center}
  \includegraphics[width=50mm]{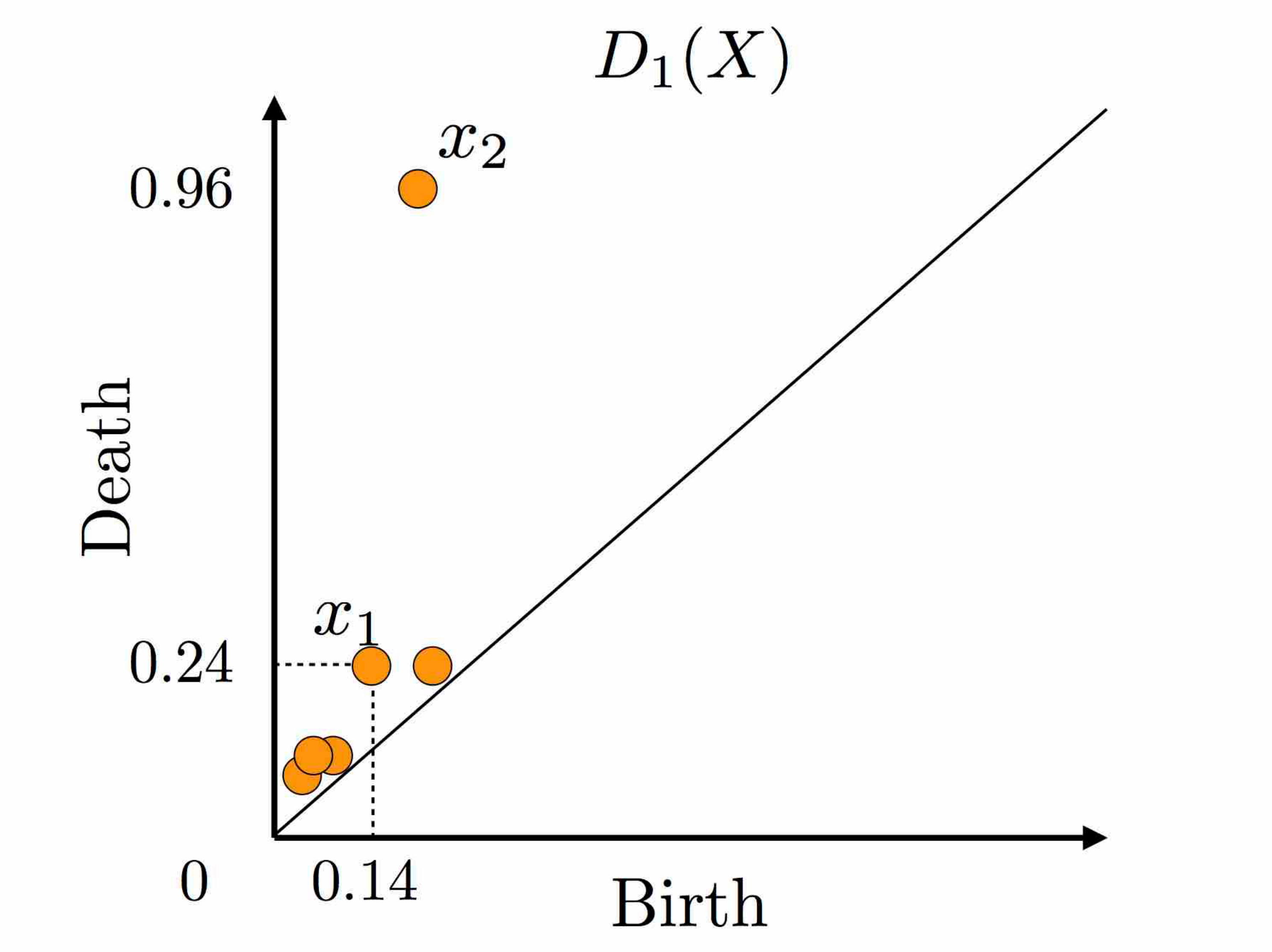}
  \vspace{-5mm}
 \end{center}
  \caption{The persistence diagram of the $r$-ball model in Figure \ref{fig:filtration}. The point $x_{1}$ represents the ring $\alpha_{1}$, which is born at $r=0.14$ and dies at $r=0.24$. The noisy rings are plotted as the points close to the diagonal.}
  \label{fig:pd}
\end{figure}

While persistence diagrams nowadays start to be applied to various problems such as the ones listed in the beginning of this section, statistical or machine learning methods for analysis on persistence diagrams are still limited.  In TDA, analysts often elaborate only one persistence diagram and, in particular, methods for handling many persistence diagrams, which should be supposed to contain randomness from the data, are at the beginning stage  (see the end of this section for related works). Hence, developing a statistical framework on persistence diagrams is a significant issue for further success of TDA.

To this aim, this paper discusses kernel methods for persistence diagrams (see Figure \ref{fig:overview}).  Since a persistence diagram is a point set of variable size, it is not straightforward to apply standard methods of statistical data analysis, which typically assume vectorial data.  Here, to vectorize persistence diagrams, we employ the framework of kernel embedding of (probability and more general) measures into reproducing kernel Hilbert spaces (RKHS). This framework has recently been developed, leading various new methods for nonparametric inference \cite{MFDS12,SGSS07,Song_etal_IEEE_SPM2013}.  It is known \cite{SFL11} that, with an appropriate choice of kernels, a signed measure can be uniquely represented by the Bochner integral of the feature vectors with respect to the measure. Since a persistence diagram can be regarded as a non-negative measure, it can be embedded into an RKHS by the Bochner integral. Once such a vector representation is obtained, we can introduce any kernel methods for persistence diagrams systematically.
\begin{figure}[htbp]
\begin{center}
\includegraphics[width=80mm]{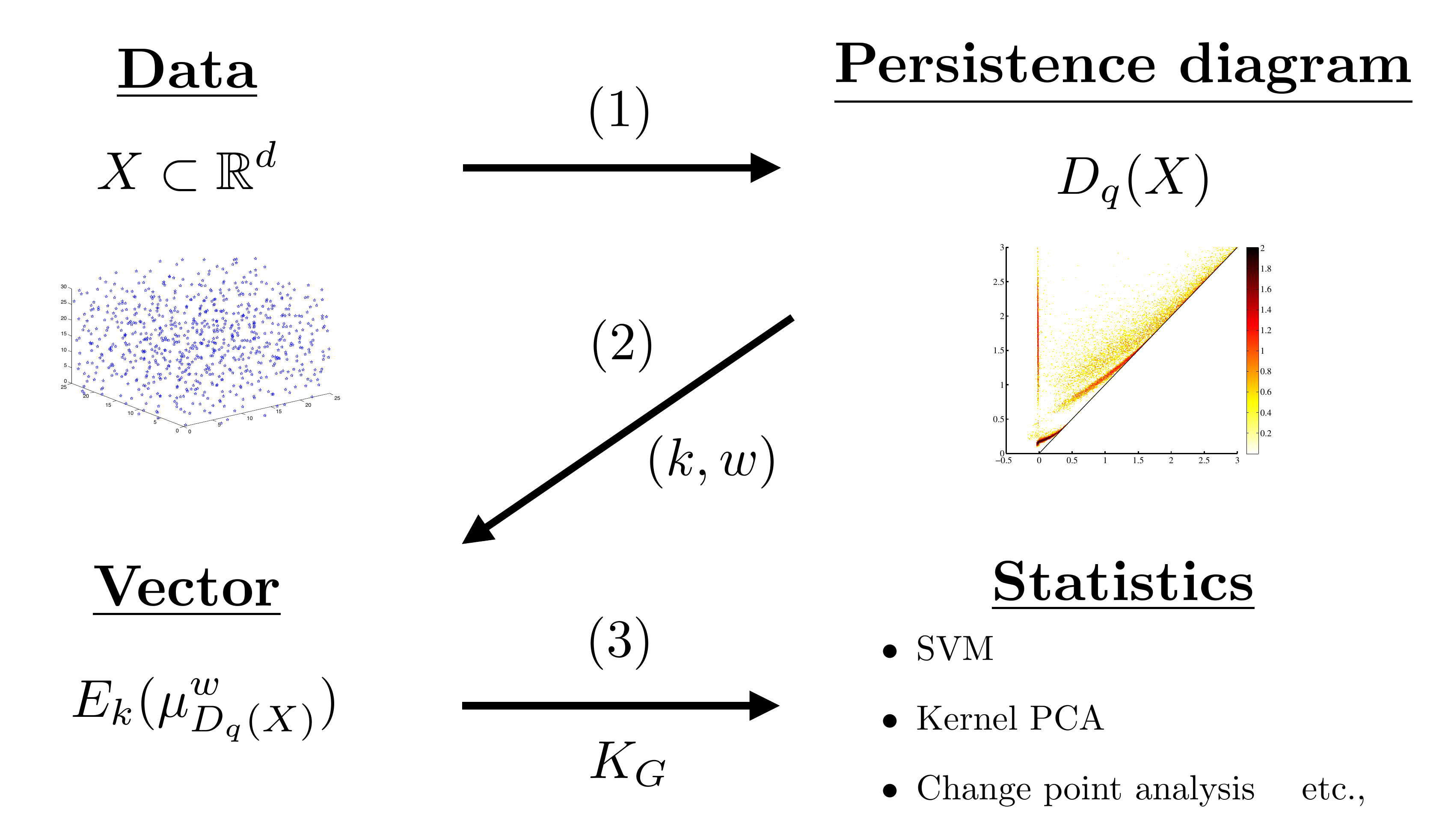}
\caption{
(1) A data $X$ is transformed into a persistence diagram $D_q(X)$ (Section \ref{subsub:persistence_diagram}).
(2) $D_q(X)$ is mapped to a vector $E_{k_{G}}(\mu_{D_{q}(X)}^{w_{{\rm arc}}})$, where $k$ is a kernel and $w$ is a weight controlling the effect of persistence (Section \ref{subsec:vectorization}).
(3) This vector provides statistical methods for persistence diagrams (Section \ref{sec:experiment}).}
\label{fig:overview}
\vspace{-6mm}
\end{center}
\end{figure}

For embedding persistence diagrams in an RKHS, we propose a useful class of positive definite kernels, called {\em persistence weighted Gaussian kernel} (PWGK).  It is important that the PWGK can discount the contributions of generators close to the diagonal (small persistence), since in many applications those generators are likely to be noise. The advantages of this kernel are as follows. (i) We can explicitly control the effect of persistence, and hence, discount the noisy generators appropriately in statistical analysis. (ii) As a theoretical contribution, the distance defined by the RKHS norm for the PWGK satisfies the stability property, which ensures the continuity from data to the vector representation of the persistence diagram. (iii) The PWGK allows efficient computation by using the random Fourier features \cite{RR07}, and thus it is applicable to persistence diagrams with a large number of generators, which are seen in practical examples (Section \ref{sec:experiment}).

We demonstrate the performance of the proposed kernel method with synthesized and real-world data, including protein datasets (taken by NMR and X-ray crystallography experiments) and oxide glasses (taken by molecular dynamics simulations).  We remark that these real-world problems have biochemical and physical significance in their own right, as detailed in Section \ref{sec:experiment}.

There are already some relevant works on statistical approaches to persistence diagrams.  Some studies discuss how to transform a persistence diagram to a vector \cite{Bu15,CMWOXW15,COO15,RT15}.  In these methods, a transformed vector is typically expressed in a Euclidean space ${\mathbb R}^k$ or a function space $L^p$, and simple and ad-hoc summary statistics like means and variances are used for data analysis such as principal component analysis and support vector machines.   The most relevant to our method is \cite{RHBK15} (see also \cite{KHNLB_NIPS2015}), where they vectorize a persistence diagram by using the difference of two Gaussian kernels evaluated at symmetric points with respect to the diagonal so that it vanishes on the diagonal.  We will show detailed comparisons between this method and ours. 
Additionally, there are some works discussing statistical properties of persistence diagrams for random data points: \cite{CGLM14} show convergence rates of persistence diagram estimation, and \cite{FLRWBS14} discuss confidence sets in a persistence diagram. These works consider a different but important direction to the statistical methods for persistence diagrams.


The remaining of this paper is organized as follows.  In Section \ref{sec:background}, we review some basics on persistence diagrams and kernel embedding methods.  In Section \ref{sec:pdkernel},  the PWGK is proposed, and some theoretical and computational issues are discussed.  Section \ref{sec:experiment} shows experimental results, and compares the proposed kernel method with other methods.

\section{Background}
\label{sec:background}

We review the concepts of persistence diagrams and kernel methods.  For readers who are not familiar with algebraic topology and homology, we give a brief summary in Supplementary material.  See also \cite{Ha01} as an accessible introduction to algebraic topology.

\subsection{Persistence diagram}
\label{subsub:persistence_diagram}

Let $X = \{\bm{x}_1, \ldots, \bm{x}_n \}$ be a finite subset in a metric space $(M,d_{M})$.
To analyze topological properties of $X$, let us consider a fattened ball model $X_{r}=\bigcup_{i=1}^{n} B(\bm{x}_{i};r)$ consisting of balls $B(\bm{x}_i;r)=\{ \bm{x} \in M \mid d_{M}(\bm{x}_i,\bm{x}) \leq r\}$ with radius $r$, and use the {\em homology} $H_{q}(X_{r})$ to describe the topology of $X_{r}$.
Here, for a topological space $S$, its {\em $q$-th homology} $H_{q}(S) \ (q=0,1,\ldots)$ is defined as a vector space, and its dimension $\dim H_{q}(S)$ counts the number of connected components $(q=0)$, rings $(q=1)$, cavities $(q=2)$, and so on\footnote{Throughout this paper we use a field coefficient for homology.}.  For the precise definition of homology, see Supplementary material.
For example, $X_{0.21}$ in Figure \ref{fig:filtration} consists of one connected component and two rings, and hence $\dim H_0(H_{0.21})=1$ and $\dim H_1(X_{0.21})=2$.

Because of $X_{r} \subset X_{s}$ for $r \leq s$, the set $\lX=\{X_{r} \mid r \geq 0\}$ becomes a filtration\footnote{A {\em filtration} is a family of subsets $\{X_a \mid a \in A \}$ indexed by a totally ordered set $A$ such that $X_a\subset X_b$ for $a\leq b$.}.
When the radius changes as in Figure \ref{fig:filtration}, a new generator $\alpha_i \in H_{q}(X_{r})$ appears at some radius $r=b_i$ and disappears at a radius $r=d_i$ larger than $b_i$ (called {\em birth} and {\em death}, respectively).
By gathering all generators $\alpha_{i} \  (i \in I)$ in the filtration $\lX$, we obtain the collection of these birth-death pairs $\underline{D}_{q}(X)=\{ (b_{i},d_{i}) \in \lR^{2} \mid i \in I \}$ as a multi-set\footnote{A {\em multi-set} is a set with multiplicity of each point.  We regard a persistence diagram as a multi-set, since several generators can have the same birth-death pairs.}.
The {\em persistence diagram} $D_{q}(X)$ is defined by the disjoint union of $\underline{D}_{q}(X)$ and the diagonal set $\DD=\{(a,a) \mid a \in \lR\}$ counted with infinite multiplicity.
A point $x=(b,d) \in D_q(X)$ is also called a {\em generator} of the persistence diagram.
The {\em persistence} $\pers(x):=d-b$ of $x$ is its lifetime and measures the robustness of $x$ in the filtration. 
We will see shortly that the diagonal set is included in a persistence diagram to simplify the definition of a distance on persistence diagrams.


Figure \ref{fig:pd} shows the persistence diagram $D_1(X)$ of $X$ given in Figure \ref{fig:filtration}. The generators $x_1$ and $x_2$ correspond to the rings $\alpha_1$ and $\alpha_2$ in Figure \ref{fig:filtration}, respectively. The persistence of $x_{2}$ is the longest, while
the other generators including $x_{1}$ have small persistences, implying that they can be seen as noisy rings.
Although there are no topological rings in $X$ itself, the persistence diagram $D_{1}(X)$ shows that there is a robust ring $\alpha_{2}$ and several noisy rings in $\lX$.
In this way, the persistence diagram provides an informative topological summary of $X_{r}$ over all $r$.

We remark that, in the finite fattened ball model, there is  only one generator in $D_0(X)$ which does not disappear in the filtration; its lifetime is $\infty$. Thus, from now on, we deal with $D_0(X)$ by removing this infinite lifetime generator in order to simplify the notation\footnote{This is called the {\em reduced persistence diagram}.}. We also note that the cardinality of $\underline{D}_{q}(X)$ obtained from the finite fattened ball model is finite.

\subsection{Stability with respect to {\boldmath $d_{B}$}}\label{sec:bottleneck_stability}

Any statistical data involve noise or stochasticity, and thus it is desired that the persistence diagrams are stable under perturbation of data.
A popular measure to study the similarity between two persistence diagrams $D$ and $E$ is the {\em bottleneck distance}
\[
d_{B}(D,E):=\inf_{\gamma} \sup_{x \in D} \norm{x-\gamma(x)}_{\infty},
\]
where $\gamma$ ranges over all multi-bijections\footnote{A {\em multi-bijection} is a bijective map between two multi-sets counted with their multiplicity.} from $D$ to $E$\footnote{For $z=(z_1,z_2)\in\lR^2$, $\Vert z \Vert_\infty$ denotes $\max(|z_1|,|z_2|)$. }. 
Note that the cardinalities of $D$ and $E$ are equal by considering the diagonal set $\Delta$ with infinite multiplicity.  
As a distance between finite sets $X,Y$ in a metric space $M$, let us recall the {\em Hausdorff distance} given by
\[
d_{H}(X,Y):= \max \left\{\sup_{\bm{x} \in X} \inf_{\bm{y} \in Y} d_{M}(\bm{x},\bm{y}), \sup_{\bm{y} \in Y} \inf_{\bm{x} \in X} d_{M}(\bm{x},\bm{y}) \right\}.
\]
Then, we have the following stability property (for more general settings, see \cite{CdSO14}).
\begin{prop}{\cite{CdSO14,CEH07}}
\label{prop:point_stability}
Let $X$ and $Y$ be finite subsets in a metric space $(M,d_{M})$.
Then the persistence diagrams satisfy
\[
d_{B}(D_{q}(X),D_{q}(Y)) \leq d_{H}(X,Y).
\]
\end{prop}

Proposition \ref{prop:point_stability} provides a geometric intuition of the stability of persistence diagrams.
Assume that $X$ is the true location of points and $Y$ is a data obtained from skewed measurement with $\ee=d_{H}(X,Y)$ (Figure \ref{fig:stability}).
If there is a point $(b,d) \in D_{q}(Y)$,
then we can find at least one generator in $X$ which is born in $(b-\ee,b+\ee)$ and dies in $(d-\ee,d+\ee)$.
Thus, the stability guarantees the similarity of two persistence diagrams, and hence 
we can infer the true topological features from one persistence diagram.
\begin{figure}[htbp]
\begin{center}
\includegraphics[width=80mm]{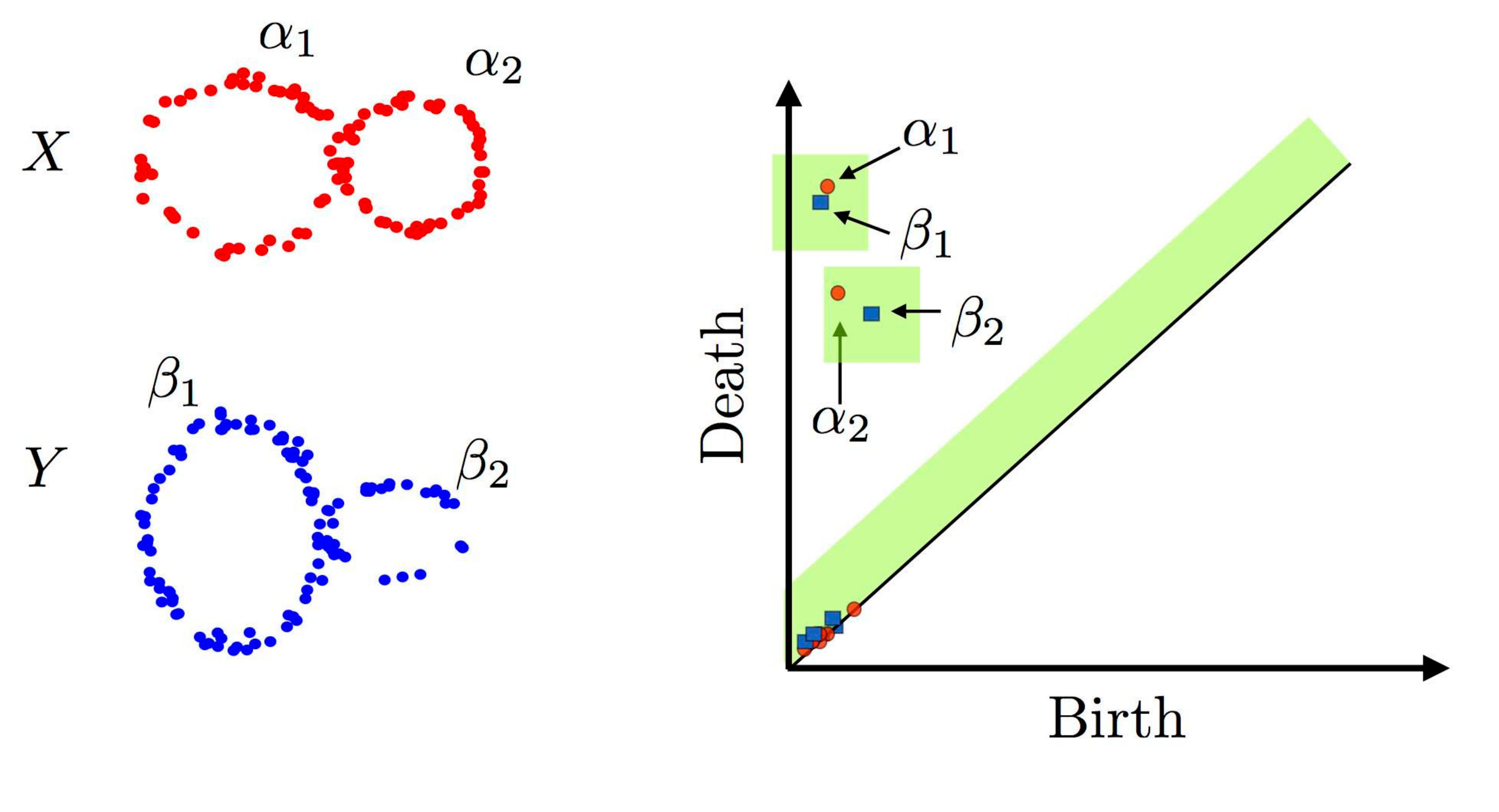}
\caption{Two data $X$ and $Y$ (left) and their persistence diagrams (right). The green region is an $\ee$-neighborhood of $D_{q}(Y)$.}
\label{fig:stability}
\end{center}
\end{figure}

\subsection{Kernel methods for representing signed measures}
\label{subsec:universal}

Let $\Omega$ be a set and $k:\Omega \times \Omega \ra \lR$ be a {\em positive definite kernel} on $\Omega$, i.e., $k$ is symmetric, and
for any number of points $x_{1},\ldots,x_{n}$ in $\Omega$, the Gram matrix $\pare{k(x_{i},x_{j})}_{i,j=1,\ldots,n}$ is nonnegative definite.
A popular example of positive definite kernel on $\lR^{d}$ is the Gaussian kernel $k_{G}(x,y)=e^{-\frac{\norm{x-y}^{2}}{2 \sigma^{2}}} \ (\sigma>0)$, where $\norm{\cdot}$ is the Euclidean norm in $\lR^{d}$.  It is also known that every positive definite kernel $k$ on $\Omega$ is uniquely associated with a reproducing kernel Hilbert space $\cH_{k}$ (RKHS).

We use a positive definite kernel to represent persistence diagrams by following the idea of the kernel mean embedding of distributions \cite{SGSS07,SFL11}.  Let $\Omega$ be a locally compact Hausdorff space, $M_{b}(\Omega)$ be the space of all finite signed Radon measures on $\Omega$, and $k$ be a bounded measurable kernel on $\Omega$.
Then we define a mapping from $M_{b}(\Omega)$ to $\cH_{k}$ by
\begin{equation}\label{eq:E_k}
E_{k}:M_{b}(\Omega) \ra \cH_{k}, ~~ \mu \mapsto \int k(\cdot, x) d \mu(x).
\end{equation}
The integral should be understood as the Bochner integral \cite{DU77}, which exists here, since $\int \norm{k(\cdot,x)}_{\cH_{k}} d \mu(x)$ is finite.

For a locally compact Hausdorff space $\Omega$,
let $C_{0}(\Omega)$ denote the space of continuous functions vanishing at infinity\footnote{A function $f$ is said to vanish at infinity if for any $\ee >0$ there is a compact set $K \subset \Omega$ such that $\sup_{x \in K^{c}} |f(x)| \leq \ee$.}.
A kernel $k$ on $\Omega$ is said to be $C_{0}$-kernel if $k(x,x)$ is of $C_{0}(\Omega)$ as a function of $x$.
If $k$ is $C_{0}$-kernel, the associated RKHS $\cH_{k}$ is a subspace of $C_{0}(\Omega)$.
A $C_{0}$-kernel $k$ is called {\em $C_{0}$-universal} if $\cH_{k}$ is dense in $C_{0}(\Omega)$.
It is known that the Gaussian kernel $k_{G}$ is $C_{0}$-universal on $\lR^{d}$ \cite{SFL11}.  When $k$ is $C_{0}$-universal, by the mapping \eqref{eq:E_k}, the vector $E_k(\mu)$ in the RKHS uniquely determines the finite signed measure $\mu$, and thus serves as a representation of $\mu$.
\begin{prop}[\cite{SFL11}]
\label{prop:C0_distance}
If $k$ is $C_{0}$-universal, the mapping $E_{k}$ is injective. Thus,
\[
d_{k}(\mu,\nu)=\norm{E_{k}(\mu)-E_{k}(\nu)}_{\cH_{k}}
\]
defines a distance on $M_{b}(\Omega)$.
\end{prop}

\section{Kernel methods for persistence diagrams}
\label{sec:pdkernel}
We propose a kernel for persistence diagrams, called the {\it Persistence Weighted Gaussian Kernel} (PWGK),  to embed the diagrams into an RKHS.  This vectorization of persistence diagrams enables us to apply any kernel methods to persistence diagrams. We show the stability theorem with respect to the distance defined by the embedding, and discuss efficient computation of the PWGK.

\subsection{Persistence weighted Gaussian kernel}
\label{subsec:vectorization}

We propose a method for vectorizing persistence diagrams using the kernel embedding \eqref{eq:E_k} by regarding a persistence diagram as a discrete measure.  In vectorizing persistence diagrams, it is important to discount the effect of generators located near the diagonal, since they tend to be caused by noise.  To this end, we explain slightly different two ways of embeddings, which turn out to introduce the same inner products for two persistence diagrams. 

First, for a persistence diagram $D$, we introduce a weighted measure $\mu^{w}_{D}:=\sum_{x \in \underline{D}} w(x)\dd_{x}$ with a weight $w(x) >0$ for each generator $x \in \underline{D}$ (Figure \ref{fig:weighted}), where $\dd_x$ is the Dirac delta measure at $x$. The weight function $w(x)$ discounts the effect of generators close to the diagonal, and a concrete choice will be discussed later.
\begin{figure}[htbp]
\begin{center}
\includegraphics[width=80mm]{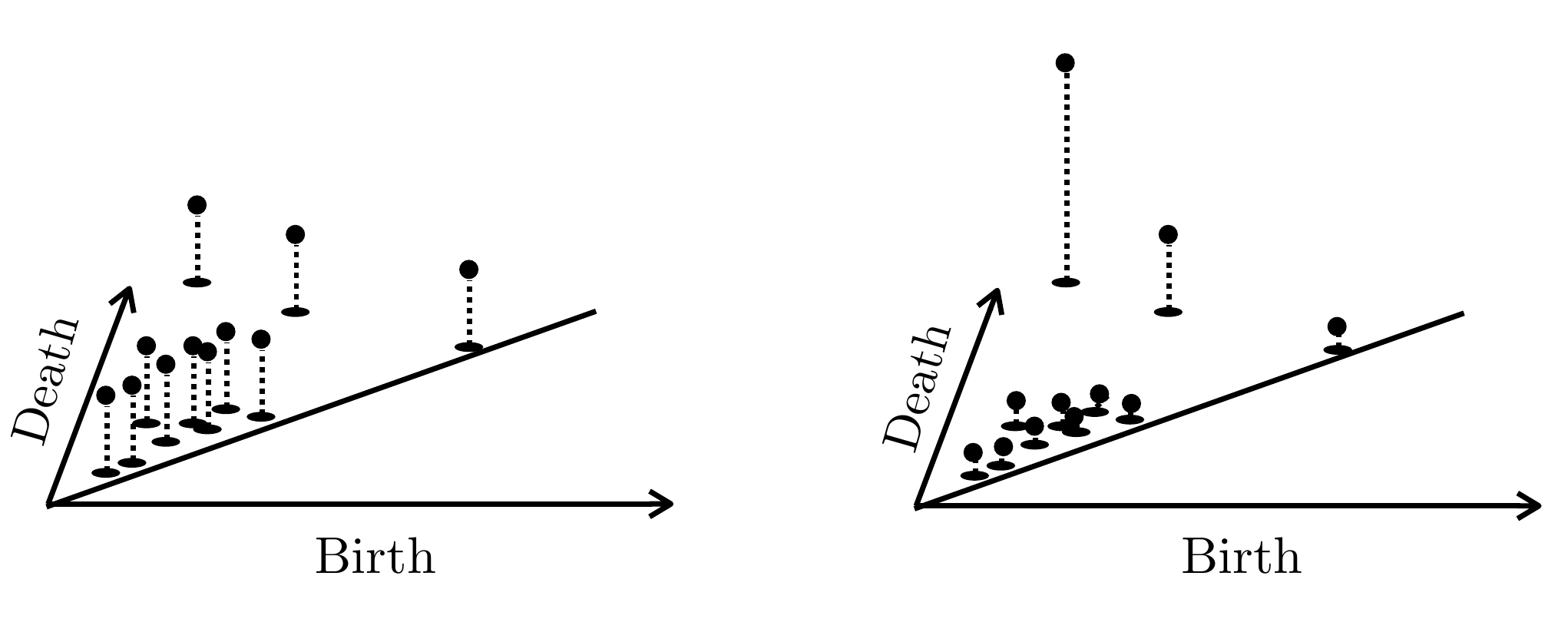}
\caption{Unweighted (left) and weighted (right) measures.}
\vspace{-3mm}
\label{fig:weighted}
\end{center}
\end{figure}
As discussed in Section \ref{subsec:universal}, given a $C_0$-universal kernel $k$ on $\lR^{2}_{ul}:=\{(b,d)\in\lR^2\mid b<d\}$, the measure $\mu^{w}_{D}$ can be embedded as an element of the RKHS $\cH_{k}$ via 
\begin{equation}\label{E_k:embed}
\mu^{w}_{D} \mapsto E_{k}(\mu^{w}_{D}):=\sum_{x\in D}w(x)k(\cdot,x).
\end{equation}
From Proposition \ref{prop:C0_distance}, this mapping does not lose any information about persistence diagrams, and $E_k(\mu^w_D)\in\cH_k$ serves as a representation of the persistence diagram.

As the second construction, let
\[
k^w(x,y):=w(x)w(y)k(x,y)
\]
be the weighted kernel with the same weight function as above, and consider the mapping 
\begin{equation}\label{E_k^w:embed}
E_{k^w}: \mu_D\mapsto \sum_{x\in D}w(x)w(\cdot)k(\cdot,x)\in \cH_{k^w}.
\end{equation}
This also defines vectorization of persistence diagrams, and it is essentially equivalent to the first one, as seen from the next proposition (See Supplementary material for the proof.).

\begin{prop}\label{prop:iso}
The following mapping
\[
\cH_k\to \cH_{k^w}, \quad f\mapsto wf
\]
defines an isomorphism between the RKHSs.  Under this isomorphism, $E_k(\mu_D^w)$ and $E_{k^w}(\mu_D)$ are identified. 
\end{prop}

Note that under the identification of Proposition \ref{prop:iso}, we have 
\[
\langle E_k(\mu_D^w), E_k(\mu_E^w)\rangle_{\cH_k} = \langle E_{k^w}(\mu_D), E_{k^w}(\mu_E)\rangle_{\cH_{k^w}}, 
\]
and thus the two constructions introduce the same similarity (and hence distance) among persistence diagrams. 
We apply methods of data analysis to vector representations $E_k(\mu_D^w)$ or $E_{k^w}(\mu_D)$.  The first construction may be more intuitive by the direct weighting of a measure, while the second one is also practically useful since all the parameter tuning is reduced to kernel choice.

For a practical purpose, we propose to use the Gaussian kernel $k_{G}(x,y)=e^{-\frac{\norm{x-y}^{2}}{2 \sigma^{2}}} \ (\sigma>0)$ for $k$, and $w_{{\rm arc}}(x)=\arctan(C \pers(x)^{p}) \ (C,p >0)$ for a weight function. The corresponding positive definite kernel is
\begin{equation}
k_{PWG}(x,y)=w_{{\rm arc}}(x)w_{{\rm arc}}(y)e^{-\frac{\norm{x-y}^{2}}{2 \sigma^{2}}}.
\end{equation}
We call it {\em Persistence Weighted Gaussian Kernel} (PWGK).
Since the Gaussian kernel is $C_{0}$-universal and $w>0$ on $\lR^2_{ul}$,
$d^{w_{{\rm arc}}}_{k_{G}}(D,E):=\Vert E_{k_{G}}(\mu_D^{w_{{\rm arc}}})-E_{k_{G}}(\mu_E^{w_{{\rm arc}}})\Vert_{\cH_{k_{G}}}$ defines a distance on the persistence diagrams.
We also note that $w_{{\rm arc}}$ is an increasing function with respect to persistence. Hence, a noisy (resp. essential) generator $x$ gives a small (resp. large) value $w_{{\rm arc}}(x)$. By adjusting the parameters $C$ and $p$, we can control the effect of the persistence.

\subsection{Stability with respect to {\boldmath $d^{w_{\rm arc}}_{k_G}$}}
\label{subsec:stability}
Given a data $X$, we vectorize the persistence diagram $D_q(X)$ as an element $E_{k_{G}}(\mu_{D_{q}(X)}^{w_{{\rm arc}}})$ of the RKHS. Then, for practical applications, this map $X \mapsto E_{k_{G}}(\mu_{D_{q}(X)}^{w_{{\rm arc}}})$ should be stable with respect to perturbations to the data as discussed in Section \ref{sec:bottleneck_stability}.
The following theorem shows that the map has the desired property (See Supplementary material for the proof.).

\begin{thm}
\label{thm:kernel_stability}
Let $M$ be a compact subset in $\lR^{d}$, $X,Y \subset M$ be finite subsets and $p>d+1$.
Then
\[
d^{w_{{\rm arc}}}_{k_{G}}(D_{q}(X),D_{q}(Y) ) \leq L(M,d;C,p,\sigma) d_{H}(X,Y),
\]
where $L(M,d;C,p,\sigma)$ is a constant depending on $M,d,C,p,\sigma$.
\end{thm}

Let $\cP_{{\rm finite}}(M)$ be the set of finite subsets in a compact subset $M \subset \lR^{d}$.
Since the constant $L(M,d;C,p,\sigma)$ is independent of $X$ and $Y$, Theorem \ref{thm:kernel_stability} concludes that the map
\[
\cP_{{\rm finite}}(M) \ra \cH_{k_{G}} , ~~ X \mapsto E_{k_{G}}(\mu_{D_{q}(X)}^{w_{{\rm arc}}})
\]
is Lipschitz continuous.  
To the best of our knowledge, a similar stability result has not been obtained for the other Gaussian type kernels (e.g., \cite{RHBK15} does not deal with the Hausdorff distance.). Our stability result is achieved by incorporating the weight function $w_{\rm arc}$ with appropriate choice of $p$.

\subsection{Kernel methods on RKHS}

Once  persistence diagrams are represented by the vectors in an RKHS, we can apply any kernel methods  to those vectors.  
The simplest choice is to consider the linear kernel
\begin{align*}
K_{L}(D,E)&=\inn{E_{k_{G}}(\mu_{D}^{w_{{\rm arc}}})}{E_{k_{G}}(\mu_{E}^{w_{{\rm arc}}})}_{\cH_{k_{G}}} \\
&=\sum_{x \in \underline{D}} \sum_{y \in \underline{E}} w_{{\rm arc}}(x)w_{{\rm arc}}(y)k_{G}(x,y)
\end{align*}
on the RKHS.  We can also consider a nonlinear kernel on the RKHS, such as the Gaussian
kernel:
\begin{align}
K_{G}(D, E) = \exp \pare{- \frac{d^{w_{{\rm arc}}}_{k_{G}}(D,E)^{2} }{ 2 \tau^{2} } },  \label{eq:kernel}
\end{align}
where $\tau$ is a positive parameter and 
\begin{align*}
d^{w_{{\rm arc}}}_{k_G}(D,E)^2& :=\norm{E_{k_G}(\mu_D^{w_{{\rm arc}}}) -E_{k_G}(\mu_E^{w_{{\rm arc}}})}_{\cH_{k_{G}}}^2 \\
& = \sum_{x \in \underline{D}} \sum_{x' \in \underline{D}} w_{{\rm arc}}(x)w_{{\rm arc}}(x')k_{G}(x,x') \\
& ~~ +\sum_{y \in \underline{E}} \sum_{y' \in \underline{E}} w_{{\rm arc}}(y)w_{{\rm arc}}(y')k_{G}(y,y') \\
& ~~ - 2 \sum_{x \in \underline{D}} \sum_{y \in \underline{E}} w_{{\rm arc}}(x)w_{{\rm arc}}(y)k_{G}(x,y).
\end{align*}

Note that we can observe better performance with nonlinear kernels for some complex tasks \cite{MFDS12}.  In this paper, we mainly apply the RKHS Gaussian kernel $K_G$ in the experimental section.
In Section \ref{sec:experiment}, we apply SVM, kernel PCA, and kernel change point detection.

\subsection{Computation of Gram matrix}
\label{subsec:calculation}

Let $\cD=\{D_{\ell} \mid \ell=1,\ldots,n\}$ be a collection of persistence diagrams.
In many practical applications, the number of generators in a persistence diagram can be large, while $n$ is often relatively small:
in Section \ref{subsec:glass}, for example, the number of generators is 30000, while $n=80$.

If the persistence diagrams contain at most $m$ points, each element of the Gram matrix $(K_{G}(D_{i},D_{j}))_{i,j=1,\ldots,n}$ involves $O(m^2)$ evaluation of $e^{-\frac{\norm{x-y}^{2}}{2\sigma^{2}}}$, resulting the complexity $O(m^{2}n^{2})$ for obtaining the Gram matrix.
Hence, reducing computational cost with respect to $m$ is an important issue, since in many applications $n$ is relatively small

We solve this computational issue by using the random Fourier features \cite{RR07}.  To be more precise, let $z_{1},\ldots,z_{M}$ be random variables from the $2$-dimensional normal distribution $N( (0,0), \sigma^{-2} I)$ where $I$ is the identity matrix.  This method approximates $e^{-\frac{\norm{x-y}^{2}}{2\sigma^{2}}}$ by $\frac{1}{M}\sum_{a=1}^{M} e^{\sqrt{-1}z_{a}x} (e^{\sqrt{-1}z_{a}y})^{*}$, where $*$ denotes the complex conjugation.
Then,  $\sum_{x \in \underline{D}_{i}} \sum_{y \in \underline{D}_{j}} w(x)w(y)k_{G}(x,y)$ is approximated by $\frac{1}{M}\sum_{a=1}^{M}B^{a}_{i} (B^{a}_{j})^{*}$, where $B^{a}_{\ell}=\sum_{x \in \underline{D}_{\ell}} w(x) e^{\sqrt{-1}z_{a}x}$.
As a result, the computational complexity of the approximated Gram matrix is $O(mnM+n^2M)$.

We note that approximation by the random Fourier features can be sensitive to the choice of $\sigma$.
If $\sigma$ is much smaller than $\norm{x-y}$, the relative error can be large.  For example,
in the case of $x=(1,2),y=(1,2.1)$ and $\sigma=0.01$, $e^{-\frac{\norm{x-y}^{2}}{2\sigma^{2}}}$ is about $10^{-22}$ while we observed the approximated value can be about $10^{-3}$ with $M=10^{3}$.
As a whole, these $m^{2}$ errors may cause a critical error to the approximation.
Moreover, if $\sigma$ is largely deviated from the ensemble $\norm{x-y}$ for $x \in \underline{D}_{i},y \in \underline{D}_{j}$, then most values $e^{-\frac{\norm{x-y}^{2}}{2 \sigma^{2}}}$ become close to $0$ or $1$.

In order to obtain a good approximation and extract meaningful values, choice of parameter is important.  For supervised learning such as SVM, we use the cross-validation (CV) approach.  For unsupervised case, we follow the heuristics proposed in \cite{GFTSSS07}.
In Section \ref{subsec:glass}, we set $\sigma =\median \{ \sigma(D_{\ell}) \mid \ell=1,\ldots,n\}$, where $\sigma(D)=\median \{ \norm{x_{i}-x_{j}} \mid x_{i},x_{j} \in \underline{D}, \ i<j \}$, so that $\sigma$ takes close values to many $\norm{x-y}$.
For the parameter $C$, we also set $C=( \median \{\pers(D_{\ell}) \mid \ell=1,\ldots,n\} )^{-p}$, where  $\pers(D)=\median \{ \pers(x_{i}) \mid x_{i} \in \underline{D} \}$.
Similarly, $\tau$ is defined by $\median \{ d^{w_{{\rm arc}}}_{k_{G}}(D_{i},D_{j}) \mid 1 \leq i < j \leq n\}$.
In this paper, since all points of data are in $\lR^{3}$, we set $p=5$ from the assumption $p>d+1$ in Theorem \ref{thm:kernel_stability}.

\section{Experiments}
\label{sec:experiment}

We demonstrate the performance of the PWGK using synthesized and real data.  In this section, all persistence diagrams are $1$-dimensional (i.e., rings) and computed by CGAL \cite{DLY15} and PHAT \cite{BKRW14}.

\subsection{Comparison to the persistence scale space kernel}
\label{subsec:pssk}
 
The most relevant work to our method is \cite{RHBK15}.  They propose a positive definite kernel called {\em persistence scale space kernel} (PSSK for short) $K_{PSS}$ on the persistence diagrams:
\begin{align*}
K_{PSS}(D,E)
&=\inn{\Phi_{t}(D)}{\Phi_{t}(E)}_{L^{2}(\lR^2_{ul})} \\
&=\frac{1}{8 \pi t} \sum_{x \in \underline{D}} \sum_{y \in \underline{E}}  e^{-\frac{ \norm{x-y}^{2} }{8 t}} - e^{-\frac{ \norm{x-\bar{y}}^{2} }{8 t}},
\end{align*}
where $\Phi_{t}(D)(x)=\frac{1}{4\pi t} \sum_{y \in \underline{D}} e^{-\frac{\norm{x-y}^{2}}{4 t}} - e^{-\frac{\norm{x-\bar{y}}^{2}}{4 t}}$, and $\bar{y}=(y^{2},y^{1})$ for $y=(y^{1},y^{2})$.
Note that the PSSK also takes zero on the diagonal by subtracting the Gaussian kernels for $y$ and $\bar{y}$.

While both methods discount noisy generators, the PWGK has the following advantages over the PSSK.  (i) The PWGK can control the effect of the persistence by $C$ and $p$ in $w_{{\rm arc}}$ independently of the bandwidth parameter $\sigma$ in the Gaussian factor, while in the PSSK only one parameter $t$ must control the global bandwidth and the discounting effect.  (ii) The approximation by the random Fourier features is not applicable to the PSSK, since it is not shift-invariant in total.  
We also note that, in \cite{RHBK15}, only the linear kernel is considered on the RKHS, while our approach involves a nonlinear kernel on the RKHS. 

Regarding the approximation of the PSSK, Nystr\"{o}m method \cite{WS01} or incomplete Cholesky factorization \cite{FS01} can be applied.  In evaluating the kernels, we need to calculate $(k(x,y))_{x\in \underline{D}_i, y\in \underline{D}_j}$, which is not symmetric.  We then need to apply Nystr\"{o}m or incomplete Cholesky to the symmetric but big positive definite matrix $(k(x,y))_{\substack{x\in \underline{D}_i,y\in\underline{D}_j\\ i,j=1,\dots,n}}$ of size $O(nm)$.
Either, we need to apply incomplete Cholesky to the non-symmtric matrix $(k(x,y))_{x\in \underline{D}_i,y\in \underline{D}_j}$ for all the combination of $(i,j)$, which requires considerable computational cost for large $n$.  In contrast, the random Fourier features can be applied to the kernel {\em function} irrespective to evaluation points, and the same Fourier expansion can be applied to any $(i,j)$. This guarantees the efficient computational cost. 

The detailed comparisons will be experimentally verified in Sections \ref{subsec:Synthesized} and \ref{subsec:glass}.  
With respect to the parameter $t$ in the PSSK, since $e^{-\frac{ \norm{x-\bar{y}}^{2} }{8 t}}$ is only used to vanish the value of the feature map $\Phi_{t}(D)$ on the diagonal,  we set $t=\frac{\sigma^{2}}{4}$ by using the same $\sigma$ defined in Section \ref{subsec:calculation}.

\subsection{Classification with synthesized data}
\label{subsec:Synthesized}

We first use the proposed method for a classification task with SVM, and compare the performance with the PSSK.  The synthesized data are generated as follows.  Each data set assumes one or two circles, and data points are located at even spaces along the circle(s).  It always contains one larger circle $S_1$ of radius $r_1^o$ raging from 1 to 10, and it may have a smaller circle $S_2$ of radius 0.2 (10 points) with probability 1/2.  Roughly speaking, the class label $Y$ is made by ${\bf XOR}(z_0, z_1)$, where $z_i$ ($i=0,1)$ is a binary variable: $z_0 =1$ if the smaller $S_2$ exists, and $z_1=1$ if the birth and death $(b^o,d^o)$ of the generator corresponding to $S_1$ satisfies $b^o<A_B$ and $d^o>A_D$ for fixed thresholds $A_B,A_D$.  We can control $b^o$ and $d^o$ by choosing the number of points $N_1^o$ along $S_1$ (for birth) and the radius $r_1^o$ (for death).  To generate the data points, we add noise effects to make the classification harder: the radius $r_1$ and sample size $N_1$ are in fact given by adding noise to $r_1^0$ and $N_1^o$, and $S_1$ are made according to the shifted $r_1$ and $N_1$, while the class label is given by the non-shifted $r_1^o$ and $N_1^o$.  For the precise description of the data generation procedure, see Supplementary material.  By this construction, the classifier needs to look at both of the location of the generator and the existence of the generator for the smaller one around the diagonal.

SVMs are trained with persistence diagrams given by 100 data sets, and evaluated with 99 independent test data sets.  For the kernel on RKHS, we used both of the linear and Gaussian kernels.  The hyper-parameters $(\sigma, C)$ in the PWGK and $t$ in the PSSK are chosen by the 10-fold cross-validation, and the degree $p$ in the weight of the PWGK is set to be $5$.  The variance parameter in the RKHS-Gaussian kernel is set by the median heuristics.  We also apply the Gaussian kernel (without any weights) for embedding persistence diagrams to RKHS.   
\begin{table}
  \caption{Results of SVM with PWGK, PSSK, and Gaussian. Average classification rates ($\%$) for 99 test data sets are shown.}\label{table:Synth_results}
  \centering
\begin{tabular}{c|c c}
  \hline
   & RKHS-Linear  & RKHS-Gauss  \\ \hline
  PWGK &  60.0  & 83.0  \\
  PSSK &  49.5  & 54.5 \\
  Gauss & 57.6  & 69.7    \\
  \hline
\end{tabular}

\end{table}
In Table \ref{table:Synth_results}, we can see that the PSSK does not work well for this problem, even worse than the Gaussian kernel, and the classification rate by the linear RKHS kernel used originally in \cite{RHBK15} is almost the chance level.  This must be caused by the difficulty in handling the global location of generators and close look around the diagonal simultaneously.  This classification task involves strong nonlinearity on the RKHS, as seen in the large improvement by PWGK+Gauss kernel.

\subsection{Analysis of ${\rm {\bf SiO_2}}$}
\label{subsec:glass}

In this experiment, we compare the PWGK and the PSSK to the non-trivial problem of glass transition on ${\rm SiO}_2$, focusing also on their computational efficiency.

When we rapidly cool down the liquid state of ${\rm SiO}_2$, it avoids the usual crystallization and changes into a glass state. 
Understanding the liquid-glass transition is an important issue for the current physics and industrial applications \cite{GS07}. For estimating the glass transition temperature by simulations, we first prepare atomic configurations of ${\rm SiO_{2}}$ for a certain range of temperatures, and then draw the temperature-enthalpy graph. 
The graph consists of two lines in high and low temperatures with slightly different slopes which correspond to the liquid and the glass states, respectively, and the glass transition temperature is conventionally estimated as an interval of the transient region combining these two lines
 (e.g., see \cite{El90}). However, since the slopes of the two lines are close to each other, determining the interval is a subtle problem, and usually the rough estimate of the interval is only available. Hence, it is desired to develop a mathematical framework to detect the glass transition temperature. 

Our strategy is to regard the glass transition temperature as the change point and detect it from a collection $\cD=\{D_{\ell} \mid \ell=1,\ldots,n\}$ of persistence diagrams made by atomic configurations of ${\rm SiO}_2$, where $\ell$ is the index of the temperatures listed in the decreasing order. We use the kernel Fisher discriminant ratio $\KFDR_{n,\ell,\gamma}(\cD)$ \cite{HMB09}
as a statistical quantity for the change point detection. Here, we set $\gamma=10^{-3}$ in this paper, and the index $\ell$ achieving the maximum of $\KFDR_{n,\ell,\gamma}(\cD)$ corresponds to the estimated change point.
The $\KFDR$ is calculated by the Gram matrix $(K(D_{i},D_{j}))_{i,j=1,\ldots,n}$ with respect to the kernel $K$.

We compute $\cD$ with $n=80$ from the data used in \cite{NHHEMN15,NHHEN15}. 
Since the persistence diagrams of ${\rm SiO_{2}}$ contain huge amount of points, we apply the random Fourier features and the Nystr{\"o}m methods \cite{DM05} for the approximations of the PWGK with the Gaussian RKHS and the PSSK, respectively. The sample sizes used in both approximations are denoted by $M$  and $c$, where $c$ is the number of chosen columns. 
Figrue \ref{fig:cpa} summarizes the plots of the change points for several sample sizes and the computational time. 
\begin{figure}
\begin{minipage}{0.48\hsize}
 \begin{center}
  \includegraphics[width=40mm]{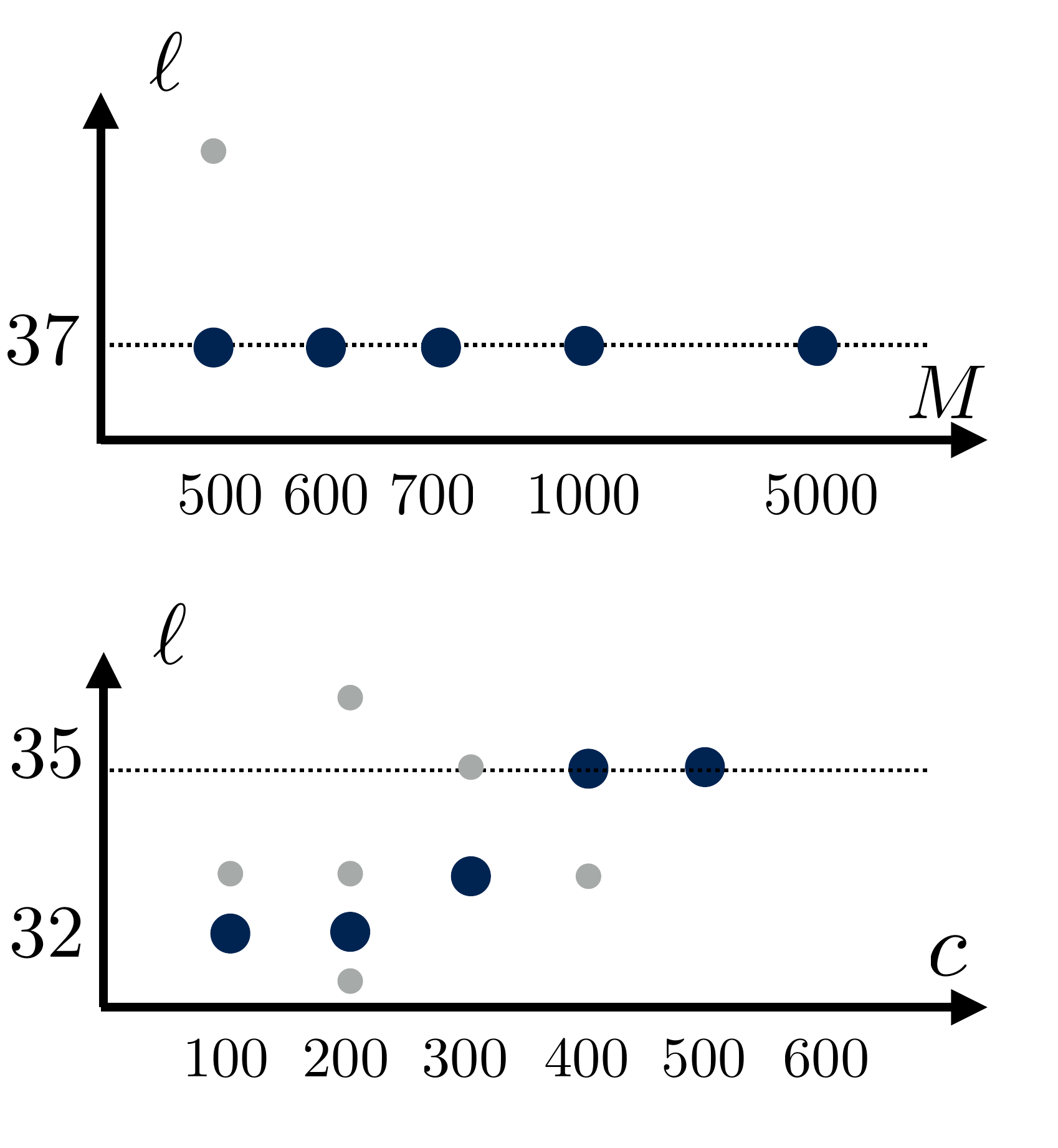}
 \end{center}
 \end{minipage}
 \begin{minipage}{0.48\hsize}
  \begin{center}
  \begin{tabular}{c|c||c|c}
  \multicolumn{4}{c}{Computational time} \\ \hline
  \multicolumn{2}{c||}{PWGK} & \multicolumn{2}{c}{PSSK} \\ \hline
  M &  sec & c & sec  \\ \hline
  500 &  40  & 100 & 130 \\
   700 &  60  &  200 & 350 \\
  1000 &  90  & 300 & 700 \\
  3000 &   220 & 400 & 1200 \\
  5000 & 410  & 500 & 1530    \\
  \hline
\end{tabular}
\end{center}
 \end{minipage}
  \caption{(Left) Estimated change points by the approximated PWGK (top) and the PSSK (bottom). The parameters $M$ and $c$ are the sample numbers used in both approximations. At each parameter, both methods are tested multiple times and the black dots and the gray dots mean the majority and the minority of estimated change points, respectively. (Right) Computational time. }
  \label{fig:cpa}
\end{figure}
The interval of the glass transition temperature $T$ estimated by the conventional method  explained above is $2000K\leq T\leq 3500K$, which corresponds to $35\leq \ell \leq 50$.


The computational complexity of the random Fourier features with respect to the sample size is $O(M)$, while that of the Nystr{\"o}m method involves matrix inversion of $O(c^3)$. For this reason, the PSSK with $c>500$ cannot be performed in reasonable time, and hence we cannot check the convergence of the change points with respect to the sample size as shown in Figure \ref{fig:cpa}. On the other hand, the PWGK plot shows the convergence to $\ell=37$, implying that $\ell=37$ is the true change point.  We here emphasize that the computation to obtain $\ell=37$ by the PWGK is much faster than the PSSK. 

\begin{figure}[htbp]
  \begin{center}
  \includegraphics[width=80mm]{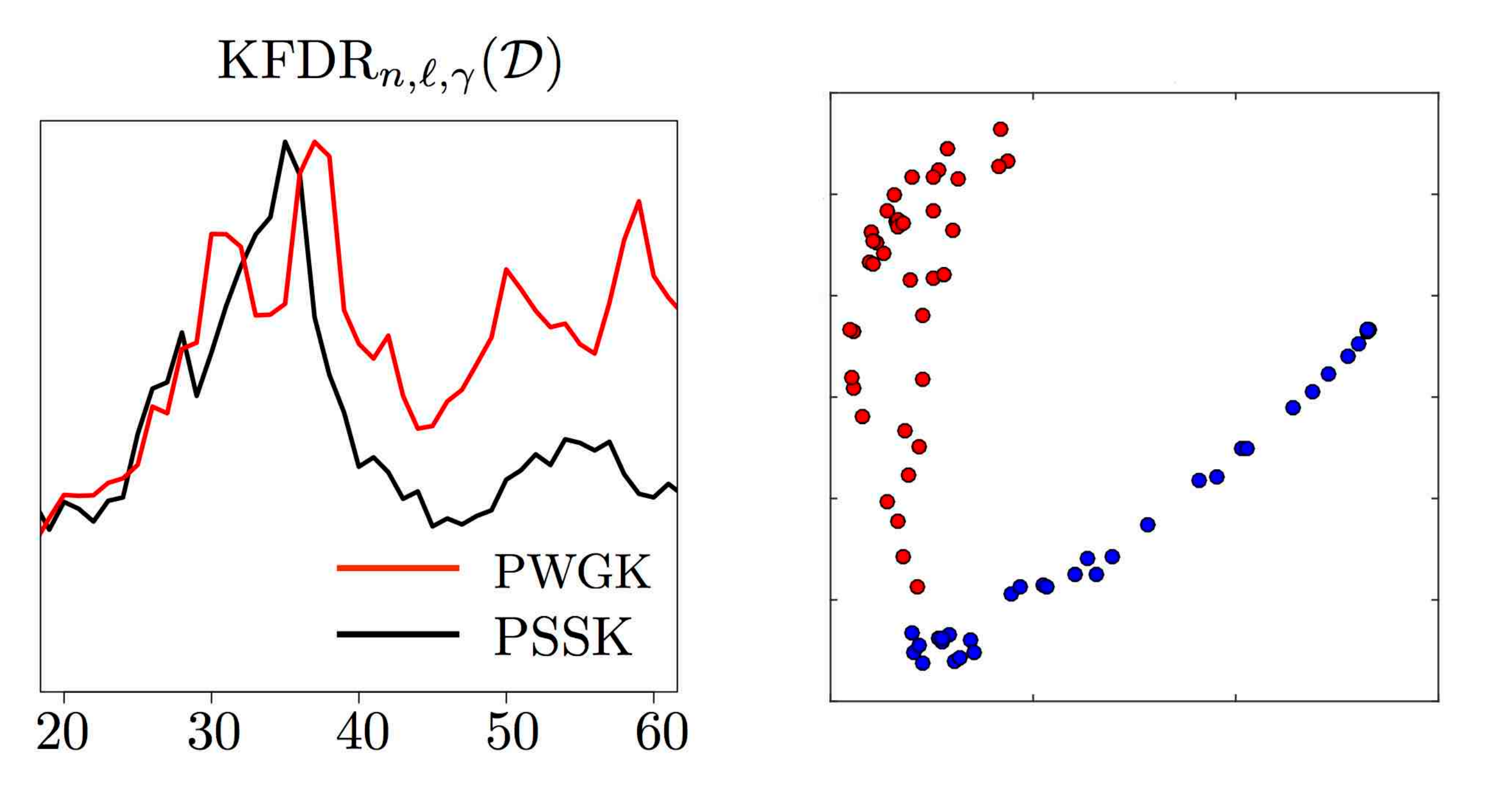}
   \caption{(Left) The $\KFDR$ plots of the PWGK $(M=1000)$ and the PSSK $(c=500)$. (Right) The $2$-dimensional KPCA plot of the PWGK.
   } 
  \label{fig:kfdr_pca}
 \end{center}
\end{figure}
Figure \ref{fig:kfdr_pca} shows the normalized plots of $\KFDR_{n,\ell,\gamma}(\cD)$ and the $2$-dimensional plot given by KPCA (the color is given by the result of the change point detection by the PWGK).  
As we see from the figure, the KPCA plot shows the clear phase change between before (red) and after (blue) the change point. This strongly suggests that the glass transition occurs at the detected change point.

\subsection{Protein classification}
\label{subsec:Protein}


We apply the PWGK to two classification tasks studied in \cite{CMWOXW15}.  They use the molecular topological fingerprint (MTF) as a feature vector for the input to the SVM.  The MTF is given by the 13 dimensional vector whose elements consist of the persistences of some specific generators (e.g., the longest, second longest, etc.) in persistence diagrams. 
We compare the performance of the PWGK with the Gaussian RKHS kernel and the MTF method under the same setting of the SVM reported in \cite{CMWOXW15}.

The first task is a protein-drug binding problem, and we classify the binding and non-binding of drug to the M2 channel protein of the influenza A virus.  For each form, 15 data were obtained by NMR experiments, in which 10 data are used for training and the remaining for testing. We randomly generated 100 ways of partitions, and calculated the classification rates.

In the second problem, the taut and relaxed forms of hemoglobin are to be classified.  For each form, 9 data were collected by the X-ray crystallography.  We select one data from each class for testing, and use the remaining  for training. All the 81 combinations are performed to calculate the CV classification rates.

The results of the two problems are shown in Table \ref{table:Protein_results}. We can see that the PWGK  achieves better performance than the MTF in both problems.
\begin{table}[ttt]
  \caption{CV classification rates ($\%$) of SVM with PWGK and MTF (cited from \cite{CMWOXW15}).} \label{table:Protein_results}
  \centering
\begin{tabular}{c|c c}
  \hline
   & Protein-Drug  &  Hemoglobin \\ \hline
  PWGK  &  100  & 88.90  \\
  MTF-SVM  &  (nbd) 93.91 / (bd) 98.31 & 84.50 \\
  \hline
\end{tabular}

\end{table}

\section{Conclusion}
In this paper, we have proposed a kernel framework for analysis with persistence diagrams, and the persistence weighted Gaussian kernel as a useful kernel for the framework.
As a significant advantage, our kernel enables one to control the effect of persistence in data analysis.
We have also proven the stability result with respect to the kernel distance.
Furthermore, we have analyzed the synthesized and real data by using the proposed kernel.
The change point detection, the principal component analysis, and the support vector machine using the PWGK derived meaningful results in physics and biochemistry. From the viewpoint of computations, our kernel provides an accurate and efficient approximation to compute the Gram matrix, suitable for practical applications of TDA.

%
%


\newpage

\section*{Supplementary Material}

This supplementary material provides a brief introduction of topological tools used in the paper, the proof of Theorem 3.1, the proof of Proposition 3.1, and explanations of synthesized data used in Section 4.2.
In oder to prove Theorem 3.1, we introduce sub-level sets in Section \ref{sec:sublevel} and total persistence in Section \ref{sec:total}, and we will prove a generalization of Theorem 3.1 in Section \ref{sec:stability}.

\appendix

\section{Topological tools}
This section summarizes some topological tools used in the paper. In general, topology concerns geometric properties invariant to continuous deformations. To study topological properties algebraically, simplicial complexes are often considered as basic objects. We start with a brief explanation of simplicial complexes, and gradually increase the generality from simplicial homology to singular and persistent homology. For more details, see \cite{Ha01}.

\subsection{Simplicial complex}\label{sec:sc}
We first introduce a combinatorial geometric model called simplicial complex to define homology. Let $P=\{1,\dots,n\}$ be a finite set (not necessarily points in a metric space). A {\em simplicial complex} with the vertex set $P$ is defined by a collection $S$ of subsets in $P$ satisfying the following properties:
\begin{enumerate}
\item $\{i\}\in S$ for $i=1,\dots,n$, and
\item if $\sigma\in S$ and $\tau\subset \sigma$, then $\tau\in S$.
\end{enumerate}

Each subset $\sigma$ with $q+1$ vertices is called a $q$-simplex. We denote the set of $q$-simplices by $S_q$. A subcollection $T\subset S$ which also becomes a simplicial complex (with possibly less vertices) is called a subcomplex of $S$.

We can visually deal with a simplicial complex $S$ as a polyhedron by pasting simplices in $S$ into a Euclidean space. The simplicial complex obtained in this way is called a geometric realization, and its polyhedron is denoted by $|S|$. 
In this context, the simplices with small $q$ correspond to points ($q=0$), edges ($q=1$), triangles ($q=2$), and tetrahedra ($q=3$). 
\begin{exam}\label{exam:sc}
{\rm 
Figure \ref{fig:sc} shows two polyhedra of simplicial complexes
\begin{align*}
&S=\{
\{1\},
\{2\},
\{3\},
\{1,2\},
\{1,3\},
\{2,3\},
\{1,2,3\}\},\\
&T=\{
\{1\},
\{2\},
\{3\},
\{1,2\},
\{1,3\},
\{2,3\}\}.
\end{align*}

\begin{figure}[htbp]
 \begin{center}
  \includegraphics[width=50mm]{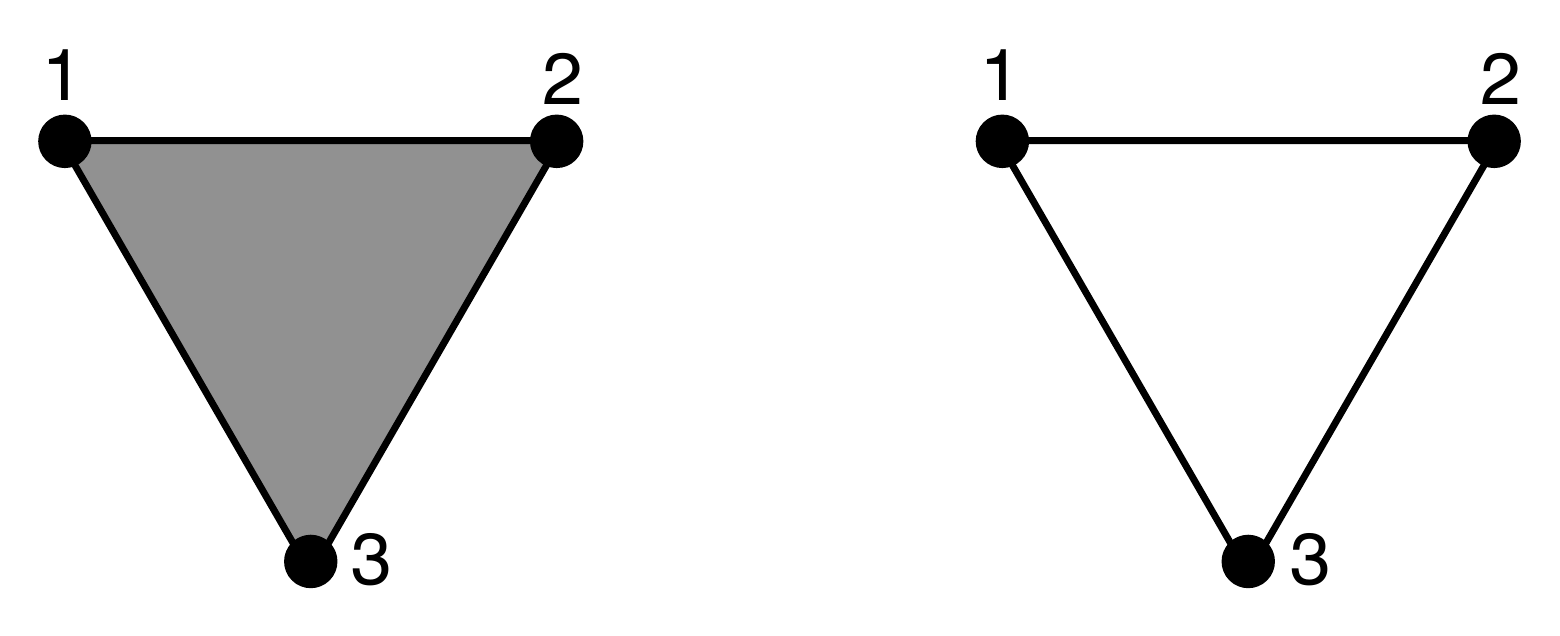}
 \end{center}
  \caption{The polyhedra of the simplicial complexes $S$ (left) and $T$ (right).}
  \label{fig:sc}
\end{figure}
}\end{exam}

\subsection{Homology}\label{sec:homology}
\subsubsection{Simplicial homology}\label{sec:simplicial_homology}
The procedure to define homology is summarized as follows:
\begin{enumerate}
\item Given a simplicial complex $S$, build a chain complex  $C_*(S)$. This is an algebraization of $S$ characterizing the boundary. 
\item Define homology by quotienting out  certain subspaces in $C_*(S)$ characterized by the boundary. 
\end{enumerate}

We begin with the procedure 1 by assigning orderings on simplices. 
When we deal with a $q$-simplex $\sigma=\{i_0,\dots,i_q\}$ as an ordered set, there are $(q+1)!$ orderings on $\sigma$. For $q>0$, we define an equivalence relation $i_{j_0},\dots,i_{j_q}  \sim  i_{\ell_0},\dots,i_{\ell_q} $ on two orderings of $\sigma$ such that they are mapped to each other by even permutations. 
By definition, two equivalence classes exist, and each of them is called an oriented simplex. 
An oriented simplex is denoted by $\langle i_{j_0},\dots,i_{j_q} \rangle$, and its opposite orientation is expressed by adding the minus $-\langle i_{j_0},\dots,i_{j_q} \rangle$.
We write $\langle\sigma\rangle=\langle i_{j_0},\dots,i_{j_q} \rangle$ for the equivalence class including $i_{j_0}<\dots<i_{j_q}$. For $q=0$, we suppose that we have only one orientation for each vertex. 

Let $K$ be a field. We construct a $K$-vector space $C_q(S)$ as 
\begin{align*}
C_q(S)={\rm Span}_K\{\langle \sigma\rangle\mid \sigma \in S_q\}
\end{align*}
for $S_q\neq\emptyset$ and $C_q(S)=0$ for $S_q=\emptyset$.
Here, ${\rm Span}_K(A)$ for a set $A$ is a vector space over $K$ such that the elements of $A$ formally form a basis of the vector space.
Furthermore, we define a linear map called the {\em boundary map} $\partial_q: C_q(S)\rightarrow C_{q-1}(S)$ by 
the linear extension of 
\begin{align}\label{eq:boundary}
\partial_q\langle i_0,\dots,i_q\rangle=\sum_{\ell=0}^q(-1)^\ell\langle i_0,\dots,\widehat{i_\ell},\dots,i_q\rangle,
\end{align}
where $\widehat{i_\ell}$ means the removal of the vertex $i_\ell$. We can regard the linear map $\partial_q$ as algebraically capturing the $(q-1)$-dimensional boundary of a $q$-dimensional object. 

For example, the image of the $2$-simplex $\langle\sigma\rangle=\langle 1,2,3\rangle$ is given by 
$\partial_2\langle\sigma\rangle=\langle2,3\rangle-\langle1,3\rangle+\langle1,2\rangle$, which is the boundary of $\sigma$ (see Figure \ref{fig:sc}). 

In practice, by arranging some orderings of the oriented $q$- and $(q-1)$- simplices, we can represent the boundary map as a matrix $M_q=(M_{\sigma,\tau})_{\sigma\in S_{q-1},\tau\in S_q}$ with the entry $M_{\sigma,\tau}=0,\pm 1$ given by the coefficient in \eqref{eq:boundary}. For the simplicial complex $S$ in Example \ref{exam:sc}, the matrix representations $M_1$ and $M_2$ of the boundary maps are  given by 
\begin{align}\label{eq:matrix}
M_2=\left[
\begin{array}{r}
1\\
1\\
-1
\end{array}\right],\quad
M_1=\left[
\begin{array}{rrr}
-1& 0&-1\\
 1&-1&0\\
 0& 1&1
\end{array}
\right]
\end{align}
Here the $1$-simplices (resp. $0$-simplices) are ordered by $\langle 1,2\rangle, \langle2,3\rangle, \langle 1,3\rangle$ (resp. $\langle1\rangle$, $\langle2\rangle$, $\langle3\rangle$).

We call a sequence of the vector spaces and linear maps
\begin{align*}
\xymatrix{
\cdots\ar[r] & C_{q+1}(S)\ar[r]^{\partial_{q+1}} & C_q(S)\ar[r]^{\partial_q} & C_{q-1}(S)\ar[r] & \cdots
}
\end{align*}
the {\em chain complex} $C_*(S)$ of $S$. As an easy exercise, we can show $\partial_{q}\circ \partial_{q+1}=0$. Hence, the subspaces $Z_q(S)={\rm ker}\partial_q$ and $B_q(S)={\rm im}\partial_{q+1}$ satisfy $B_q(S)\subset Z_q(S)$. Then, the {\em $q$-th (simplicial) homology} is defined by taking the quotient space
\begin{align*}
H_q(S)=Z_q(S)/B_q(S).
\end{align*}
Intuitively, the dimension of $H_q(S)$ counts the number of $q$-dimensional holes in $S$ and each generator of the vector space $H_q(S)$ corresponds to these holes. We remark that the homology as a vector space is independent of the orientations of simplices. 

For a subcomplex $T$ of $S$, the inclusion map $\rho:T\hookrightarrow S$ naturally induces a linear map in homology $\rho_q: H_q(T)\rightarrow H_q(S)$. Namely, an element $[c]\in H_q(T)$ is mapped to $[c]\in H_q(S)$, where the equivalence class $[c]$ is taken in each vector space. 

For example, the simplicial complex $S$ in Example \ref{exam:sc} has 
\[
Z_1(S)={\rm Span}_K[\begin{array}{ccc}1 & 1 & -1\end{array}]^T=B_1(S)
\]
 from (\ref{eq:matrix}). Hence $H_1(S)=0$, meaning that there are no $1$-dimensional hole (ring) in $S$.
On the other hand, since $Z_1(T)=Z_1(S)$ and $B_1(T)=0$, we have $H_1(T)\simeq K$, meaning that $T$ consists of one ring. Hence, the induced linear map $\rho_1: H_1(T)\rightarrow H_1(S)$ means that the ring in $T$ disappears in $S$ under $T\hookrightarrow S$.

A topological space $X$ is called {\em triangulable} if there exists a geometric realization of a simplicial complex $S$ whose polyhedron is homeomorphic\footnote{A continuous map $f:X \ra Y$ is said to be {\em homeomorphic} if $f:X \ra Y$ is bijective and the inverse $f^{-1}:Y \ra X$ is also continuous.} to $X$. For such a triangulable topological space, the homology is defined by $H_q(X)=H_q(S)$. This is well-defined, since a different geometric realization provides an isomorphic homology. 

\subsubsection{Singular homology}\label{sec:singular_homology}
We here extend the homology to general topological spaces. Let $e_0,\dots,e_q$ be the standard basis of ${\mathbb R}^{q+1}$ (i.e., $e_i=(0,\dots,0,1,0,\dots,0)$, 1 at  $(i+1)$-th position, and 0 otherwise), and set
\begin{align*}
&\Delta_q=\rl{
\sum_{i=0}^q\lambda_ie_i \lmid \sum_{i=0}^q\lambda_i=1, \lambda_i\geq 0
},\\
&\Delta^\ell_q=\rl{
\sum_{i=0}^q\lambda_ie_i \lmid \sum_{i=0}^q\lambda_i=1, \lambda_i\geq 0, \lambda_\ell=0
}.
\end{align*}
We also denote the inclusion by $\iota^\ell_q: \Delta^\ell_q\hookrightarrow \Delta_q$. 

For a topological space $X$, a continuous map $\sigma: \Delta_q\rightarrow X$ is called a singular $q$-simplex, and let $X_q$ be the set of $q$-simplices. 
We construct a $K$-vector space $C_q(X)$ as 
\begin{align*}
C_q(X)={\rm Span}_K\{\sigma\mid\sigma\in X_q\}.
\end{align*}
The boundary map $\partial_q:C_q(X)\rightarrow C_{q-1}(X)$ is defined by the linear extension of
\begin{align*}
\partial_q \sigma=\sum_{\ell=0}^q(-1)^\ell\sigma\circ \iota^\ell_q.
\end{align*}

Even in this setting, we can show that $\partial_{q}\circ\partial_{q+1}=0$, and hence the subspaces $Z_q(X)={\rm ker}\partial_q$ and $B_q(X)={\rm im}\partial_{q+1}$ satisfy $B_q(X)\subset Z_q(X)$. Then, the {\it $q$-th (singular) homology} is similarly defined by 
\begin{align*}
H_q(X)=Z_q(X)/B_q(X).
\end{align*}
It is known that, for a triangulable topological space, the homology of this definition is isomorphic to that defined in 
\ref{sec:simplicial_homology}. From this reason, we hereafter identify simplicial and singular homology. 

The induced linear map in homology for an inclusion pair of topological space $Y\subset X$ is similarly defined as in \ref{sec:simplicial_homology}. 

\subsection{Persistent homology}\label{sec:homology}
Let $\mathbb{X}=\{X_a\mid a\in\mathbb{R}\}$ be a (right continuous) filtration of topological spaces, i.e., $X_a\subset X_b$ for $a\leq b$ and $X_{a}=\bigcap_{a<b}X_{b}$. For $a\leq b$, we denote the linear map induced from $X_a\hookrightarrow X_b$ by $\rho^b_a: H_q(X_a)\rightarrow H_q(X_b)$. The {\em persistent homology} $H_q(\mathbb{X})=(H_q(X_a),\rho^b_a)$ of $\mathbb{X}$ is defined by the family of homology $\{H_q(X_a)\mid a\in\mathbb{R}\}$ and the induced linear maps $\rho^b_a$ for all $a\leq b$. 

A {\em homological critical value} of $H_q(\mathbb{X})$ is the number $a\in\mathbb{R}$ such that the linear map $\rho^{a+\ee}_{a-\ee}: H_q(X_{a-\ee})\rightarrow H_q(X_{a+\ee})$ is not isomorphic for any $\ee>0$. The persistent homology $H_q(\mathbb{X})$ is called {\em tame}, if $\dim H_q(X_a)<\infty$ for any $a\in\mathbb{R}$ and the number of homological critical values is finite. 
A tame persistent homology $H_q(\mathbb{X})$ has a nice decomposition property:
\begin{thm}[\cite{ZC05}]\label{thm:decomposition}
A tame persistent homology can be uniquely expressed by
\begin{align}\label{eq:decom}
H_q(\mathbb{X})\simeq\bigoplus_{i\in I} I[b_i,d_i],
\end{align}
where $I[b_i,d_i]=(U_a,\iota^b_a)$ consists of a family of vector spaces
\begin{align*}
U_a=\left\{\begin{array}{ll}
K,&b_i\leq a < d_i\\
0,&{\rm otherwise}
\end{array}\right.,
\end{align*}
and the identity map $\iota^b_a={\rm id}_{K}$ for $b_i\leq a \leq b<d_i$.
\end{thm}

Each summand $I[b_i,d_i]$ is called a generator of the persistent homology and $(b_i,d_i)$ is called its birth-death pair. We note that, when $\dim H_q(X_a) \neq 0$ for any $a<0$ (or for any $a>0$ resp.), the decomposition (\ref{eq:decom}) should be understood in the sense that some $b_i$ takes the value $-\infty$ (or $d_i=\infty$, resp.), where $-\infty,\infty$ are the elements in the extended real $\overline{\mathbb{R}}=\mathbb{R}\cup\{-\infty,\infty\}$. 

From the decomposition in Theorem \ref{thm:decomposition}, we define a multiset
\begin{align*}
\underline{D}_q(\mathbb{X})=\rl{ (b_i,d_i)\in\overline{\mathbb{R}}^2 \lmid i\in I }.
\end{align*}
The {\em persistence diagram} $D_q(\mathbb{X})$ is defined by the disjoint union of $\underline{D}_q(\mathbb{X})$ and the diagonal set $\Delta=\{(a,a)\mid a\in\overline{\mathbb{R}}\}$ with infinite multiplicity.

By definition, a generator $(b,d)\in D_q(\mathbb{X})$ close to $\Delta$ possesses a short lifetime, implying a noisy topological feature under parameter changes. On the other hand, a generator far away from $\Delta$ can be regarded as a robust feature. 

As an example, we show a filtration of simplicial complexes $\mathbb{X}$ in Figure \ref{fig:filtration_suppli} and its persistence diagram $D_1(\mathbb{X})$ in Figure \ref{fig:pd}.

\begin{figure}[htbp]
 \begin{center}
  \includegraphics[width=100mm]{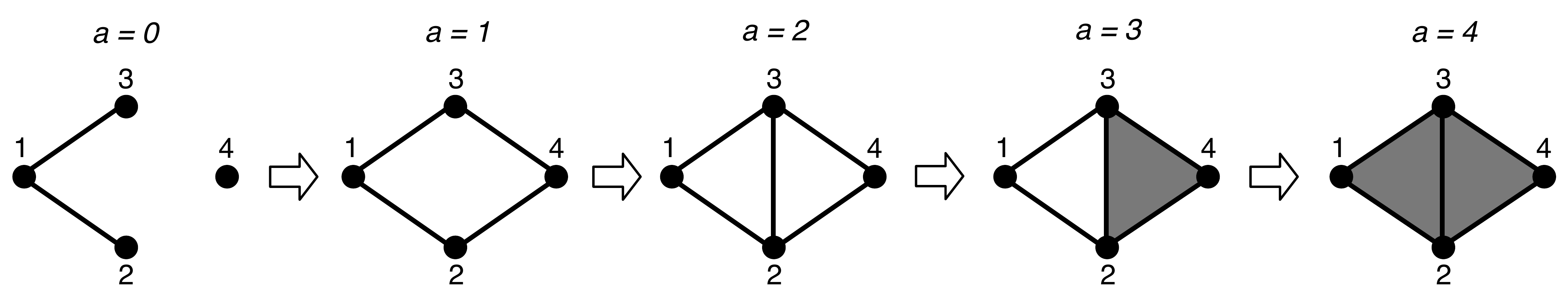}
 \end{center}
  \caption{A filtration $\mathbb{X}$ of simplicial complexes.}
  \label{fig:filtration_suppli}
\end{figure}

\begin{figure}[htbp]
 \begin{center}
  \includegraphics[width=40mm]{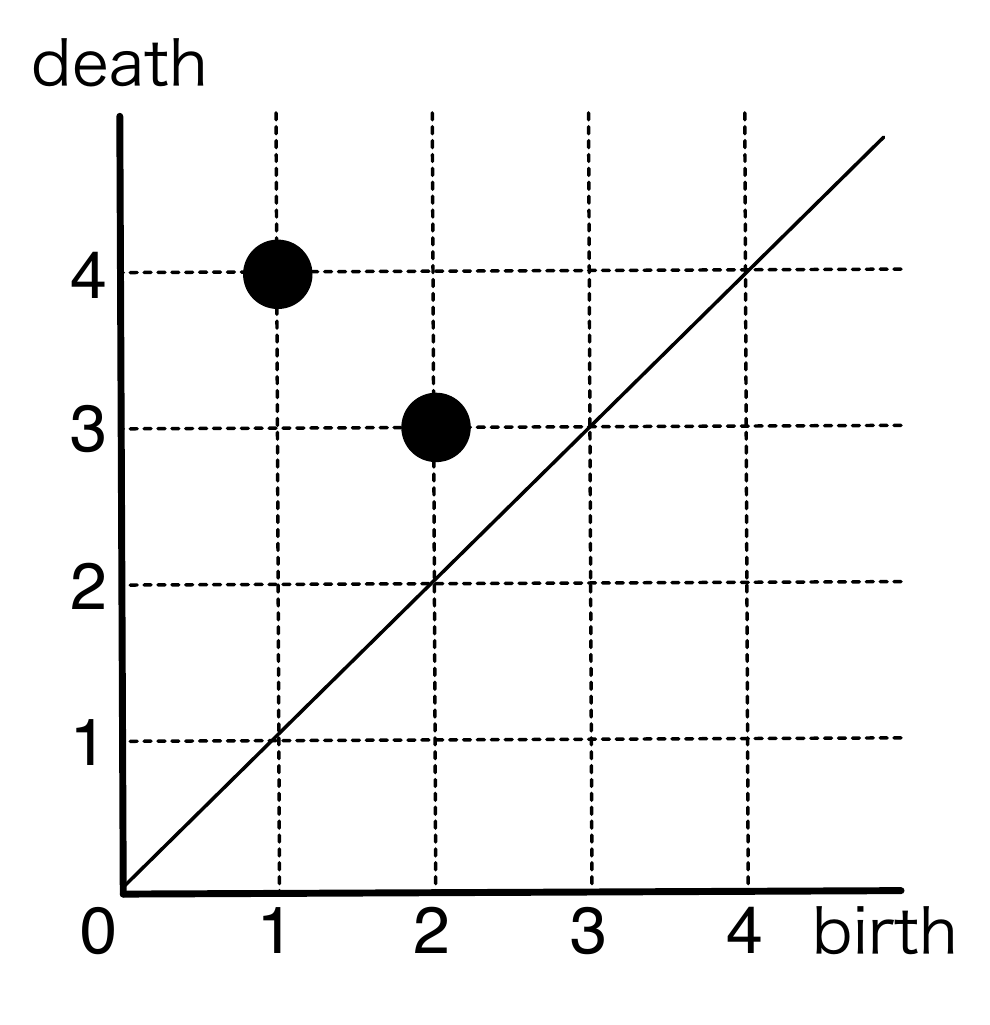}
 \end{center}
  \caption{Persistence diagram $D_1(\mathbb{X})$.}
  \label{fig:pd}
\end{figure}

In the paper, we considered a filtration $\mathbb{X}=\{X_r\mid r\in\mathbb{R}\}$ generated by a finite subset $X=\{\bm{x}_1, \ldots, \bm{x}_n\}$ in some metric space $(M,d_M)$ by $X_r=\bigcup_{i=1}^nB(\bm{x}_i;r)$, where $B(\bm{x};r)=\{\bm{x}'\in M\mid d_M(\bm{x},\bm{x}')\leq r\}$. For notational simplicity, we denote the persistence diagram of this filtration model by $D_q(X)$.
\section{Sub-level sets}
\label{sec:sublevel}

A popular way of constructing a filtration of topological spaces and thus persistent homology is to use the sub-level sets of a function.

Given a function $f:M \ra \lR$, a {\em sub-level set} $F_{a}=\{ \bm{x} \in M \mid f(\bm{x}) \leq a\}$ defines a filtration $\lF = \{ F_{a} \mid a \in \lR\}$ (Figure \ref{fig:sub-level}) and its persistent homology $H_q(\mathbb{F})$. A function $f$ is said to be tame if the persistent homology $H_q(\mathbb{F})$ is tame for all $q$.
For a tame function $f:M \ra \lR$, the persistence diagram can be defined and denoted by $D_{q}(f)$.
Note that, for a finite set $X=\{\bm{x}_1, \ldots, \bm{x}_n \}$ in $M \subset \lR^{d}$ and the function $f_{X}:M \ra \lR$ defined by $f_{X}(\bm{x})=\min_{\bm{x}_{i} \in X} d_{M} (\bm{x},\bm{x}_{i})$, $f_{X}$ is tame and the persistence diagram $D_{q}(f_{X})$ of the sub-level set is the same as $D_{q}(X)$.
\begin{figure}[htbp]
 \begin{center}
  \includegraphics[width=60mm]{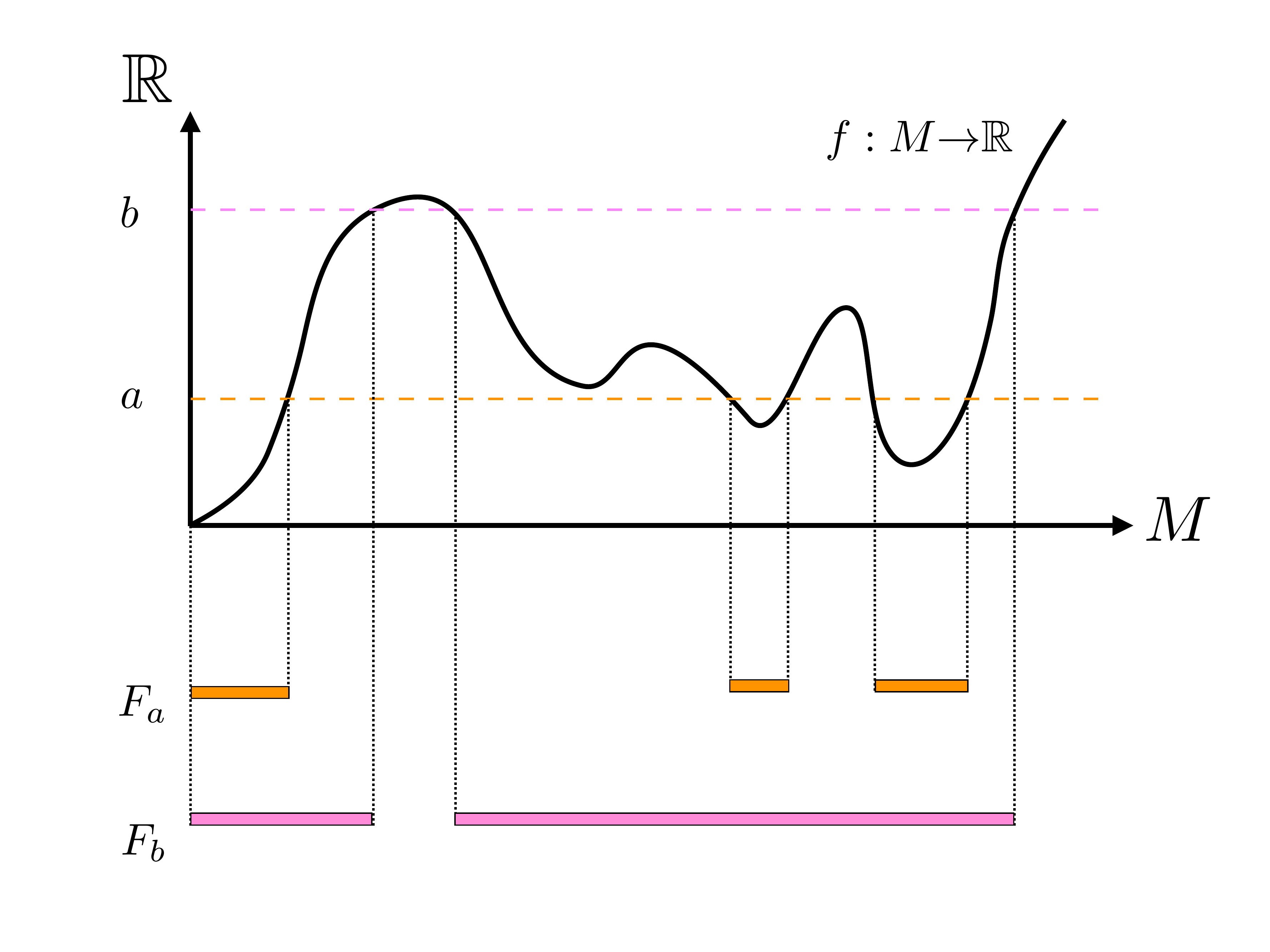}
 \end{center}
  \caption{A filtration $\lF=\{ F_{a} \mid a \in \lR\}$ of sub-level sets.}
  \label{fig:sub-level}
\end{figure}

In case of sub-level sets, the stability of the persistence diagram is shown in \cite{CEH07}.
Here, we measure the difference between two functions $f,g:M \ra \lR$ by $\norm{f-g}_{\infty}=\sup_{\bm{x} \in M} |f(\bm{x})-g(\bm{x})|$.
\begin{prop}[\cite{CEH07}]
\label{prop:function_stability}
Let $M$ be a triangulable compact metric space with continuous tame functions $f,g:M \ra \lR$. 
Then the persistence diagrams satisfy
\[
d_{B}(D_{q}(f),D_{q}(g)) \leq \norm{f-g}_{\infty}.
\]
\end{prop}

For finite subsets $X$ and $Y$ in a metric space $M$, since the difference of norm $\norm{f_{X}-f_{Y}}_{\infty}$ is nothing but the Hausdorff distance $d_{H}(X,Y)$, we obtain Proposition 2.1 as a corollary.

\section{Total persistence}
\label{sec:total}

In order to estimate constants $L$ appearing in Theorem 3.1, we will review several properties of persistence.

Let $(M,d_{M})$ be a triangulable compact metric space.
For a Lipschitz function $f:M \ra \lR$, the {\em degree-$p$ total persistence} is defined by 
\begin{equation*}
\Pers_{p}(f,t)=\sum_{\substack{x \in D_{q}(f) \\ \pers(x)>t}} \pers(x)^{p} 
\end{equation*}
for $0 \leq t \leq \Amp(f)$.
Here, $\Amp(f) := \max_{\bm{x} \in M}f(\bm{x})-\min_{\bm{x} \in M} f(\bm{x})$ is the amplitude of $f$.
Let $K$ be a triangulated simplicial complex of $M$ by a homeomorphism $\vartheta: |K| \ra M$.
The diameter of a simplex $\sigma \in K$ and the mesh of the triangulation $K$ are defined by $\diam(\sigma) = \max_{\bm{x},\bm{y} \in \sigma} d_{M}(\vartheta(\bm{x}),\vartheta(\bm{y}))$ and $\mesh(K) = \max_{\sigma \in K} \diam(\sigma )$, respectively.
Furthermore, let us set $N(r)= \min_{\mesh(K ) \leq r} \card{K}$.
The degree-$p$ total persistence is bounded as follows:
\begin{lem}[\cite{CEHM10}]
Let $M$ be a triangulable compact metric space and $f:M \ra \lR$ be a tame Lipschitz function.
Then $\Pers_{p}(f,t)$ is bounded from above by
\[
t^{p}N(t/\Lip(f)) + p \int^{\Amp(f)}_{\ee=t} N(\ee/\Lip(f))\ee^{p-1}d\ee,
\]
where $\Lip(f)$ is the Lipschitz constant of $f$.
\end{lem}

For a compact triangulable subspace $M$ in $\lR^{d}$, the number of $d$-cubes with length $r>0$ covering $M$ is bounded for every $r>0$.
The minimum of these numbers is $N(r)$ and bounded from above by $C_{M}/r^{d}$ for some constant $C_{M}$ depending only on $M$.

For $p>d$, we can find upper bounds for the both terms as follows:
\begin{align*}
t^{p}N(t/\Lip(f)) \leq t^{p}C_{M} \frac{ \Lip(f)^{d}} {t^{d}} 
\end{align*}
and
\begin{align*}
p \int^{\Amp(f)}_{\ee=t} N(\ee/\Lip(f))\ee^{p-1}d\ee \leq \frac{p}{p-d} C_{M} \Lip(f)^{d} \Amp(f)^{p-d} .
\end{align*}
We note $\lim_{t \ra 0} t^{p}N(t/\Lip(f))=0$ for $p>d$.
Then, an upper bound of the total persistence $\Pers(f):=\Pers_{p}(f,0)$ is given as follows:
\begin{lem}
\label{lem:total}
Let $M$ be a triangulable compact subspace in $\lR^{d}$ and $p>d$.
For any Lipschitz function $f:M \ra \lR$, 
\[
\Pers_{p}(f) \leq \frac{p}{p-d}C_{M} \Lip(f)^{d} \Amp(f)^{p-d},
\]
where $C_{M}$ is a constant depending only on $M$.
\end{lem}

In case of a finite subset $X \subset \lR^{d}$, there always exists an $R$-ball $M$ containing $X$ for some $R>0$, which is a triangulable compact subspace in $\lR^{d}$, and the total persistence of $f_{X}$ is bounded as follows:

\begin{lem}
\label{lem:point_total}
Let $M$ be a triangulable compact subspace in $\lR^{d}$, $X=\{\bm{x}_{1},\ldots,\bm{x}_{n}\}$ be a finite subset of $M$, and $p>d$. Then 
\[
\Pers_{p}(f_{X}) \leq \frac{p}{p-d}C_{M}\diam(M)^{p-d},
\]
where $C_{M}$ is a constant depending only on $M$.
\end{lem}

\begin{proof}
The Lipschitz constant of $f_{X}$ is $1$, because, for any $\bm{x},\bm{y} \in M$,
\begin{align*}
 f_{X}(\bm{x})-f_{X}(\bm{y}) 
&=\min_{\bm{x}_{i} \in X} d(\bm{x},\bm{x}_{i})-\min_{\bm{x}_{i} \in X} d(\bm{y},\bm{x}_{i}) \\
&\leq \min_{\bm{x}_{i} \in X} (d(\bm{x},\bm{y})+d(\bm{y},\bm{x}_{i})) - \min_{\bm{x}_{i} \in X} d(\bm{y},\bm{x}_{i}) \\
&=d(\bm{x},\bm{y}).
\end{align*}
Moreover, the amplitude $\Amp(f_{X})$ is less than or equal to $\diam(M)$, because $\min_{\bm{x} \in M} f_{X}(\bm{x})=0$ and $\max_{\bm{x} \in M} f_{X}(\bm{x}) \leq \diam(M)$.
\end{proof}

For more general, for a persistence diagram $D$, we define $p$-degree total persistence of $D$ by $\Pers_{p}(D):=\sum_{x \in D} \pers(x)^{p}$.
Let $x_{1},\ldots,x_{n}$ be points of $\underline{D}$.
Then we can consider the $n$-dimensional vector
\[
v(D):=\pare{ \pers(x_{1}),\ldots,\pers(x_{n}) }
\]
from $D$.
Since each $\pers(x) \ (x \in \underline{D})$ is always positive, by using the $\ell^{p}$-norm of $v(D)$, we have $\Pers_{p}(D)=\norm{v(D)}^{p}_{p}$.
In general, since $\norm{v}_{q} \leq \norm{v}_{p} \ (v \in \lR^{n}, \ 1 \leq p \leq q)$, we have
\[
\Pers_{q}(D)^{\frac{1}{q}} =\norm{v(D)}_{q} \leq \norm{v(D)}_{p} = \Pers_{p}(D)^{\frac{1}{p}}.
\]

\begin{prop}
\label{prop:persistence_inequality}
If $1 \leq p \leq q<\infty$ and $\Pers_{p}(D)$ is bounded, $\Pers_{q}(D)$ is also bounded.
\end{prop}

\section{Proof of Theorem 3.1 and its generalization}
\label{sec:stability}

We can have the following generalized stability result:
\begin{thm}
\label{thm:general_stability}
Let $D$ and $E$ be finite persistence diagrams.
Then
\begin{align*}
d^{w_{{\rm arc}}}_{k_{G}}(D,E) \leq L(D,E;C,p,\sigma) d_{B}(D,E),
\end{align*}
where $L(D,E;C,p,\sigma)$ is a constant depending on $D,E,C,p,\sigma$.
\end{thm}

In fact, we can calculate the constant $L(D,E;C,p,\sigma)$ in Theorem \ref{thm:general_stability} such as
\begin{align*}
\biggl\{ \frac{\sqrt{2}}{\sigma} \Pers_{p}(D)+ 2p\Pers_{p-1}(D)  +(2p+1)\Pers_{p-1}(E) \biggr\} C.
\end{align*}
Actually, this constant is dependent on $D$ and $E$, and hence we cannot say that the map $D \mapsto E_{k_{G}}(\mu^{w_{\rm arc}}_{D})$ is continuous.
However, in the case of persistence diagrams obtained from finite sets, this constant becomes independent of $D$ and $E$.
From now on, to obtain the constant $L(D,E;C,p,\sigma)$, we show several lemmas.

\begin{lem}
\label{lemm:Lip_k}
For any $x,y \in \lR^{2}$, $\norm{k_{G}(\cdot,x)-k_{G}(\cdot,y)}_{\cH_{k_{G}}} \leq \frac{\sqrt{2}}{\sigma} \norm{x-y}_{\infty}$.
\end{lem}

\begin{proof}
\begin{align}
\norm{k_{G}(\cdot,x)-k_{G}(\cdot,y)}^{2}_{\cH_{k_{G}}} \nonumber 
&=k_{G}(x,x)+k_{G}(y,y)-2k_{G}(x,y) \nonumber \\
&=1+1-2 e^{-\frac{\norm{x-y}^{2}}{2 \sigma^{2}}} \nonumber \\
&= 2 \pare{ 1-e^{-\frac{\norm{x-y}^{2}}{2 \sigma^{2}}} } \nonumber  \\
& \leq  \frac{1}{\sigma^{2}}\norm{x-y}^{2} \label{eq:eq1} \\
& \leq \frac{2}{\sigma^{2}}\norm{x-y}^{2}_{\infty}. \label{eq:eq2} 
\end{align}
We have used the fact $1-e^{-t} \leq t \ ( t \in \lR)$ in \eqref{eq:eq1} and $\norm{x}^{2} \leq 2 \norm{x}^{2}_{\infty} \ (x \in \lR^{2})$ in \eqref{eq:eq2}.
\end{proof}

\begin{lem}
\label{lemm:persistence}
For any $x,y \in \lR^{2}$, the difference of persistences $\abs{\pers(x)-\pers(y)}$ is less than or equal to $2 \norm{x-y}_{\infty}$.
\end{lem}

\begin{proof}
For $x=(x_{1},x_{2}),y=(y_{1},y_{2})$, we have
\begin{align}
| \pers(x)-\pers(y) |
&= |(x_{2}-x_{1}) - (y_{2}-y_{1})| \nonumber \\
&\leq |x_{2}-y_{2}|+|x_{1}-y_{1}| \nonumber \\
&\leq 2 \norm{x-y}_{\infty}. \nonumber
\end{align}
\end{proof}

\begin{lem}
\label{lemm:w_continuous}
For any $x,y \in \lR^{2}$, we have 
\begin{align*}
 \abs{w_{{\rm arc}}(x)-w_{{\rm arc}}(y)}  \leq 2pC \max \{\pers(x)^{p-1}, \pers(y)^{p-1}\} \norm{x-y}_{\infty}.
\end{align*}
\end{lem}

\begin{proof}
\begin{align}
&\abs{w_{{\rm arc}}(x)-w_{{\rm arc}}(y) } \nonumber \\
&=\abs{ \arctan(C \pers(x)^{p})-\arctan(C \pers(y)^{p}) } \label{eq:arctan} \\
& \leq C \abs{ \pers(x)^{p}-\pers(y)^{p}} \nonumber \\
& \leq C \abs{\pers(x)-\pers(y)} p \max \{\pers(x)^{p-1}, \pers(y)^{p-1}\} \label{eq:p}\\
& \leq 2 pC \max \{\pers(x)^{p-1}, \pers(y)^{p-1}\} \norm{x-y}_{\infty}. \nonumber
\end{align}
We have used the fact that the Lipschitz constant of $\arctan$ is $1$ in \eqref{eq:arctan} and for any $s,t \in \lR$,
\begin{align*}
s^{p}-t^{p} 
&=(s-t)(s^{p-1}+s^{p-2}t+\cdots+t^{p-1}) \\
& \leq (s-t) p \max \{ s^{p-1},t^{p-1}\}
\end{align*}
in \eqref{eq:p}.
\end{proof}

\begin{proof}[Proof of Theorem \ref{thm:general_stability}]
Here, let a multi-bijection $\gamma:D \ra E$ such that $\sup_{x \in D} \norm{x-\gamma(x)}_{\infty} \leq \ee$ and  $E':=\underline{E}-(\gamma(\underline{D})-\DD)$.
By the definition of $\gamma$, each point in $E'$ is $\ee$-close to the diagonal.
Then, we evaluate $d^{w_{{\rm arc}}}_{k_{G}}(D,E)$ as follows:


\begin{align}
&d^{w_{{\rm arc}}}_{k_{G}}(D,E) \nonumber \\
&=\norm{\int k_{G}(\cdot,x) d \mu^{w_{{\rm arc}}}_{D}(x)-\int k_{G}(\cdot,x) d \mu^{w_{{\rm arc}}}_{E}(y)}_{\cH_{k_{G}}} \nonumber \\
&\leq \norm{ \int k_{G}(\cdot,x) d \mu^{w_{{\rm arc}}}_{D}(x) -\int k_{G}(\cdot,y) d \mu^{w_{{\rm arc}}}_{\gamma(\underline{D})}(y)}_{\cH_{k_{G}}} + \norm{ \int k_{G}(\cdot,y) d \mu^{w_{{\rm arc}}}_{E'}(y)}_{\cH_{k_{G}}} \nonumber \\
&=\norm{ \sum_{x \in D} w_{{\rm arc}}(x)k_{G}(\cdot,x) -\sum_{x \in D} w_{{\rm arc}}(\gamma(x)) k_{G}(\cdot,\gamma(x))}_{\cH_{k_{G}}} \nonumber  \\
& ~ +\norm{\sum_{y \in E'} w_{{\rm arc}}(y)k_{G}(\cdot,y) }_{\cH_{k_{G}}} \nonumber  \\
&\leq \norm{ \sum_{x \in D} w_{{\rm arc}}(x)k_{G}(\cdot,x) -\sum_{x \in D} w_{{\rm arc}}(x) k_{G}(\cdot,\gamma(x)) }_{\cH_{k_{G}}} \nonumber  \\
& ~+ \norm{\sum_{x \in D} w_{{\rm arc}}(x) k_{G}(\cdot,\gamma(x)) -\sum_{x \in D} w_{{\rm arc}}(\gamma(x)) k_{G}(\cdot,\gamma(x))}_{\cH_{k_{G}}} \nonumber \\
& ~+\norm{\sum_{y \in E'} w_{{\rm arc}}(y)k_{G}(\cdot,y) }_{\cH_{k_{G}}} \nonumber \\
&\leq \sum_{x \in D} w_{{\rm arc}}(x) \norm{k_{G}(\cdot,x)-k_{G}(\cdot,\gamma(x))}_{\cH_{k_{G}}} \nonumber \\
& ~+ \sum_{x \in D} \abs{w_{{\rm arc}}(x)-w_{{\rm arc}}(\gamma(x))} \norm{k_{G}(\cdot,\gamma(x))}_{\cH_{k_{G}}} + \sum_{y \in E'} w_{{\rm arc}}(y) \norm{k_{G}(\cdot,y)}_{\cH_{k_{G}}} \label{eq:dk_halfway}
\end{align}
Form Lemma \ref{lemm:Lip_k}, we have
\[
\norm{k_{G}(\cdot,x)-k_{G}(\cdot,\gamma(x))}_{\cH_{k_{G}}} \leq \frac{\sqrt{2}}{\sigma} \norm{x-\gamma(x)}_{\infty},
\]
from Lemma \ref{lemm:w_continuous}, 
\begin{align*}
\abs{w_{{\rm arc}}(x)-w_{{\rm arc}}(\gamma(x))}  \leq 2pC \max \{\pers(x)^{p-1}, \pers(\gamma(x))^{p-1}\} \norm{x-\gamma(x)}_{\infty}.
\end{align*}
Moreover, for any $x \in \lR^{2}$ we have $\norm{k_{G}(\cdot,x)}_{\cH_{k_{G}}}=\sqrt{k_{G}(x,x)}=1$.
Therefore, the continuation of \eqref{eq:dk_halfway} is
\begin{align}
& \sum_{x \in D} w_{{\rm arc}}(x) \frac{\sqrt{2}}{\sigma} \norm{x-\gamma(x)}_{\infty} \nonumber \\
& + \sum_{x \in D} 2pC \max \{\pers(x)^{p-1}, \pers(\gamma(x))^{p-1}\} \norm{x-\gamma(x)}_{\infty}+ \sum_{y \in E'} w(y) \nonumber \\
& \leq  \frac{\sqrt{2}}{\sigma} C \ee \sum_{x \in D} \pers(x)^{p} \label{eq:warc_bound} \\
& ~+ 2pC \ee \sum_{x \in D}  \pare{\pers(x)^{p-1}+\pers(\gamma(x))^{p-1} } +C \sum_{y \in E'} \pers(y)^{p} \nonumber \\
& \leq  \frac{\sqrt{2}}{\sigma} C \ee \sum_{x \in D} \pers(x)^{p} \nonumber \\
& ~+ 2pC \ee \sum_{x \in D}  \pare{\pers(x)^{p-1}+\pers(\gamma(x))^{p-1} }  +C \ee \sum_{y \in E'} \pers(y)^{p-1} \label{eq:pers_ee} \\
& = \biggl\{ \frac{\sqrt{2}}{\sigma} \Pers_{p}(D)+ 2p\Pers_{p-1}(D) +2p\Pers_{p-1}(\gamma(D))+\Pers_{p-1}(E') \biggr\} C \ee \\
&= \biggl\{ \frac{\sqrt{2}}{\sigma} \Pers_{p}(D)+ 2p\Pers_{p-1}(D) +(2p+1)\Pers_{p-1}(E)  \biggr\} C \ee \label{eq:total_last}.
\end{align}
We have used the fact $w_{\rm arc}(x)\leq C \pers(x)^p$ in \eqref{eq:warc_bound}, $\pers(y) \leq \ee \ ( y \in E')$ in \eqref{eq:pers_ee} and 
\[
\Pers_{p-1}(\gamma(D)),\Pers_{p-1}(E') \leq \Pers_{p-1}(E)
\]
 in \eqref{eq:total_last}.
Thus, if both $(p-1)$-degree total persistence of $D$ and that of $E$ are bounded, since $p$-degree total persistence of $D$ is also bounded from Proposition \ref{prop:persistence_inequality}, the coefficient of $\ee$ appearing in \eqref{eq:total_last} is bounded.
\end{proof}

\begin{proof}[Proof of Theorem 3.1]
For any finite set $X \subset M$, from Lemma \ref{lem:point_total}, there exists a constant $C_{M}>0$ such that
\[
\Pers_{p}(D_{q}(X)) \leq \frac{p}{p-d} C_{M}\diam(M)^{p-d}.
\]
By replacing $D$ and $E$ with $D_{q}(X)$ and $D_{q}(Y)$ in \eqref{eq:total_last}, respectively, we have 
\begin{align*}
& \biggl\{ \frac{\sqrt{2}}{\sigma} \Pers_{p}(D_{q}(X))+ 2p \Pers_{p-1}(D_{q}(X))+(2p+1)\Pers_{p-1}(D_{q}(Y))  \biggr\}C\ee  \\
& \leq \biggl\{ \frac{\sqrt{2}}{\sigma} \frac{p}{p-d} C_{M}\diam(M)^{p-d} + (4p+1) \frac{p-1}{p-1-d} C_{M}\diam(M)^{p-1-d}  \biggr\} C\ee\\
& = L(M,d;C,p,\sigma) \ee.
\end{align*}
This value is a constant dependent on $M,d,C,p,\sigma$ but independent of $X,Y$.
\end{proof}

\section{Proof of Proposition 3.1}

\begin{proof}
Let 
\[
    \tilde{\cH}:=\{wf:\lR^2_{ul}\to\lR\mid f\in \cH_k\},
\]
and define its inner product by
\[
\langle wf,wg\rangle_{\tilde{\cH}}:=\langle f, g\rangle_{\cH_k}.
\]
Using $w>0$, it is easy to see that $\tilde{\cH}$ is a Hilbert space, and the mapping $f\mapsto wf$ is an isomorphism of the Hilbert spaces.  Moreover, we can see $\tilde{\cH}$ is in fact $\cH_{k^w}$.  To see this, it is sufficient to see that $k^w(\cdot,x)=w(\cdot)w(x)k(\cdot,x)$  is a reproducing kernel of $\tilde{\cH}$.  Then, the uniqueness of a reproducing kernel for an RKHS completes the proof.  The reproducing property is in fact proved from 
\[
\langle wf, k^w(\cdot,x)\rangle_{\tilde{\cH}}=\langle f,w(x)k(\cdot,x)\rangle_{\cH_k} = w(x)f(x) = (wf)(x).
\]
The second assertion is obvious from Equations (2) and (3) in Section 3.1.
\end{proof}

\section{Synthesized data}
\label{sec:synthesized}
We describe the details of synthesized data used in Section 4.2.

Each data set contains a variable number of 2 dimensional points, which are located on one or two circles.  A data set always contains points along a larger circle $S_1$ of random center $(x_1,y_1)$ and radius $r_1$ raging from 1 to 10.  $N_1$ points are located along $S_1$ at even spaces.  Another circle $S_2$ of radius $r_2 =0.2$ at the origin exists with probability 0.5, and 10 points are located at even spaces, if exists.
The variables $r_1$ and $N_1$ are in fact noisy version of $r_1^o$ and $N_1^o$, which are also randomly given.  The generative model of these variables are shown later.

Consider the persistence diagram for the $r$-ball model with the above points, and let $(b,d)$ ($b\leq d$) be the birth-death coordinate for the generator corresponding to $S_1$.  Note that there are no noise to the points on the circle, once the circle is given, and $(b,d)$ is  essentially determined by $N_1$ and $r_1$:
if $N_1$ is sufficiently large, the birth time $b$ is approximately $\pi r_1/N_1$ (here we use $\sin \theta \approx \theta$ for a small $\theta$), and the death time $d$ is equal to $r_1$, if no intersection occurs with balls around the points on $S_2$. See Figure \ref{fig:birth-death}.

For the binary classification, the class label is assigned depending on whether or not $S_2$ exists, and whether or not the hypothetical (or true) circle and virtual data points given by $r_1^o$ and $N_1^o$ has long persistence.  Let $A_B$ and $A_D$ be constants (we set $A_B=1$ and $A_D = 4$ in the experiments).
We introduce binary variables $z_0$ and $z_1$; $z_0=1$ if $S_2$ exists, and $z_1=1$ if $\pi r_1^o/N_1^o\leq A_B$ and $r_1^o\geq A_D$.  Note that $z_1=1$ if the hypothetical (or true) generator $(b^o,d^o)$ approximately satisfy $b^o\leq A_B$ and $d^o\geq A_D$.   The class label $Y$ of the data set is then given by
\[
    Y={\bf XOR}(z_0,z_1).
\]

The generative models of $r_1,N_1,r_1^o$ and $N_1^o$ are given as follows.
The radius $r_1$ is generated by
\[
	r_1 = r_1^o+W^2, \quad W\sim N(0,1),
\]
and the number of points $N_1$ is given by
\[
	N_1 = \lceil N_1^o + 2 U\rceil, \quad U\sim N(0,1).
\]
The center point is given by $(x_1,y_1)=1.5(r_1,r_1)+(V_1^2,V_2^2)$ with independent $V_1,V_2\sim N(0,2)$.
The ``true" radius $r_1^o$ is generated by
\[
    r_1^o=\max\{ 1+8T^2, 10 \}, \qquad T\sim N(0,1)
\]
and the number of points $N_1^o$ is a random integer with equal probability in $[\lceil N_*/2\rceil, 4N_*]$, where $N_*:=\pi r_1^o/A_B$. Note that $N_*$ points on the circle of radius $r_1^o$ approximately give a birth time $A_B$.  If $N_1^o > N_*$ the true birth $b^o$ is smaller than $A_B$ approximately.

Note that the class label $Y$ is given based on $r_1^o$ and $N_1^o$, while the data points and persistence diagram are based on the noisy version $r_1$ and $N_1$.
By this randomness, the best classification boundary does not give 100$\%$ classification rate. Figure \ref{fig:synthesized} shows two examples of data set.

\begin{figure}[htbp]
 \begin{center}
  \includegraphics[width=120mm]{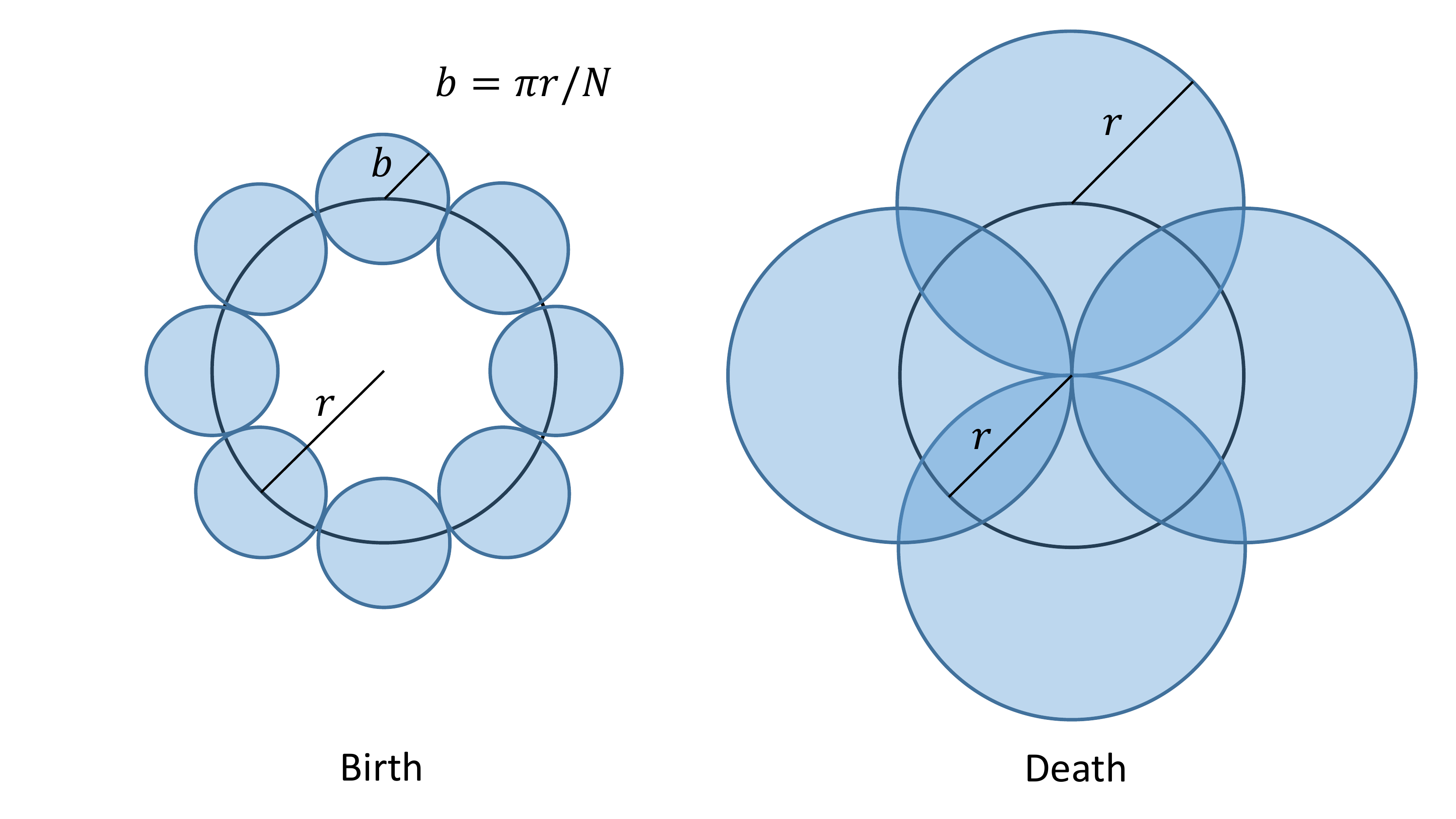}
 \end{center}
  \caption{Birth and death of the generator for $S_1$.  }
  \label{fig:birth-death}
\end{figure}

\begin{figure}[htbp]
 \begin{center}
  \includegraphics[width=120mm]{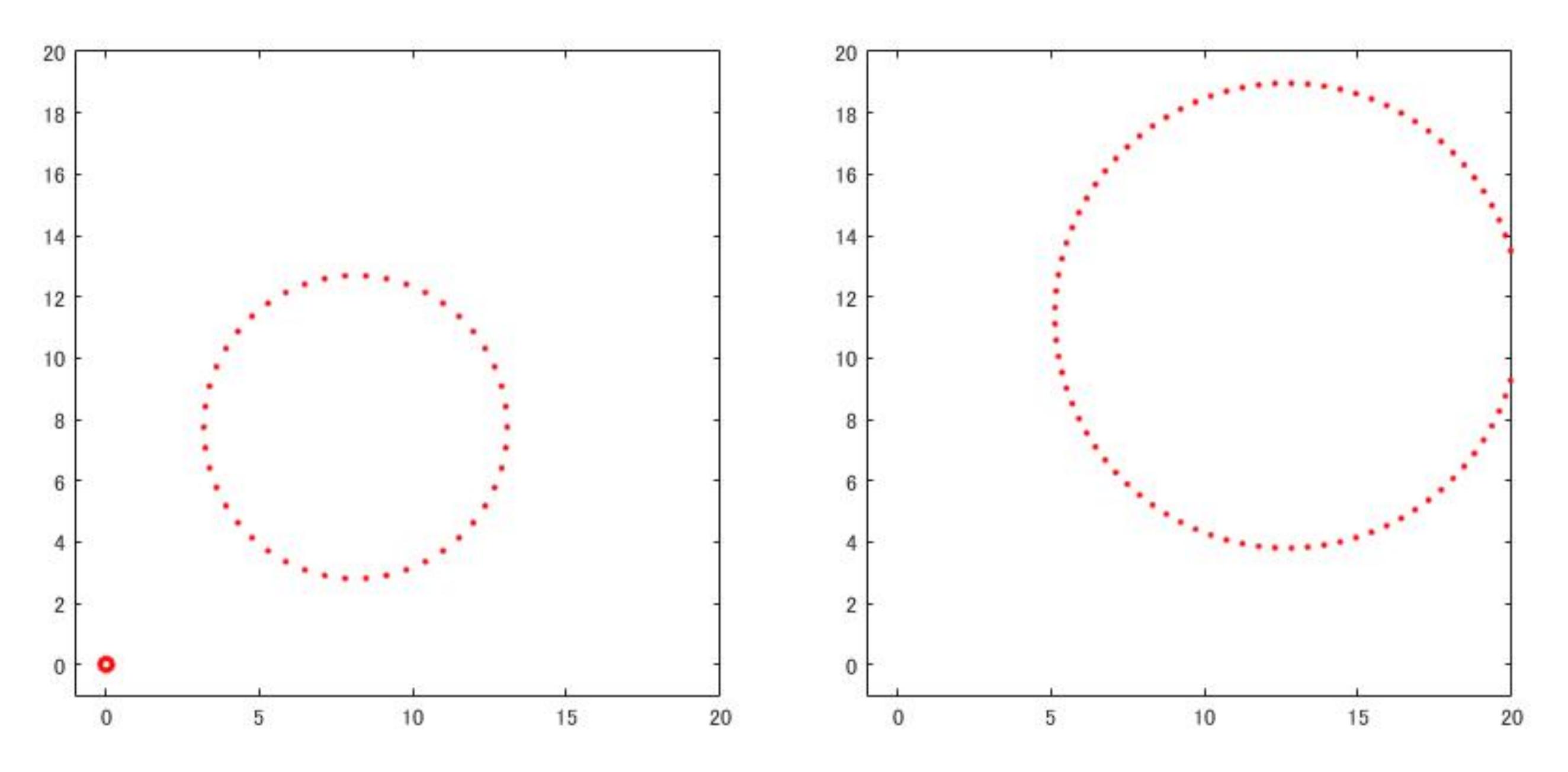}
 \end{center}
  \caption{Examples of synthesized data.  Left: $S_2$ exits. Right: No $S_2$. }
  \label{fig:synthesized}
\end{figure}

\bibliography{reference}
\bibliographystyle{alpha}

\end{document}